\documentclass[10pt]{amsart}
\usepackage{verbatim}
\usepackage{amsmath,amsfonts,amscd,amssymb,epsfig,euscript,bm}
\usepackage{amsxtra,mathrsfs}
\usepackage{pdfsync}
\newtheorem{theorem}{Theorem}[section]
\newtheorem{proposition}{Proposition}[section]
\newtheorem{lemma}{Lemma}[section]
\newtheorem{corollary}[theorem]{Corollary}
\theoremstyle{definition}
\newtheorem{definition}{Definition}{}

\theoremstyle{remark} 
\newtheorem{remark}{Remark}[section]

\newcommand{\hth}{\hat{\theta}}

\newcommand{\spbp}{\sqrt{-1}\partial \bar{\partial}}

\newcommand{\vep}{\varepsilon}
\newcommand{\al}{\alpha}
\newcommand{\la}{\lambda}

\newcommand{\ddbar}{\sqrt{-1}\partial \bar\partial}
 
\newcommand{\be}{\beta}
\newcommand{\CC}{\mathbb{C}}
\newcommand{\RR}{\mathbb{R}}
\newcommand{\hvarphi}{\hat\varphi}
\numberwithin{equation}{section}
\usepackage{hyperref}
\hypersetup{colorlinks,citecolor=blue,plainpages=false,hypertexnames=false}
\begin{document}
\title[A numerical criterion for generalised Monge-Amp\`ere equations]{A numerical criterion for generalised Monge-Amp\`ere equations on projective manifolds}
\author{Ved V. Datar}
\address{Department of Mathematics, Indian Institute of Science, Bangalore, India - 560012}
\email{vvdatar@iisc.ac.in}
\author{Vamsi Pritham Pingali}
\address{Department of Mathematics, Indian Institute of Science, Bangalore, India - 560012}
\email{vamsipingali@iisc.ac.in}
\begin{abstract} 
We prove that generalised Monge-Amp\`ere equations (a family of equations which includes the inverse Hessian equations like the $J$-equation, as well as the Monge-Amp\`ere equation) on projective manifolds have smooth solutions if certain intersection numbers are positive.  As corollaries of our work, we improve a result of Chen (albeit in the projective case) on the existence of solutions to the $J$-equation, and prove a conjecture of Sz\'ekelyhidi in the projective case on the solvability of certain inverse Hessian equations. The key new ingredient in improving Chen's result is a degenerate concentration of mass result. We also prove an equivariant version of our results, albeit under the assumption of uniform positivity. In particular, we can recover existing results on manifolds with large symmetry such as projective toric manifolds. 
\end{abstract}
\maketitle
\section{Introduction}\label{Introsec}

\indent The Monge-Amp\`ere equation and related PDE such as the J-equation \cite{Donaldson, XXchen, SongWeinkove, Weinkove, Weinkoveearlier, collinsgabor, gaborlejmi, FangLaiMa, songweinkoveb, Xiaojian, Xiao, rui, yao, hashi, songd} and the deformed Hermitian Yang-Mills (dHYM) equation \cite{jacobyau, collinsjacobyau, collinsxieyau, jacob, yama, pingthree, han, hanyama, pindhym, gchen,  Stoppa, hanjintwo, tak1, tak2, tak3, kawai} play an important role in K\"ahler geometry. Very often, their solvability hinges on the existence of a subsolution (which, in a sense, functions as a ``barrier"). Unlike the Monge-Amp\`ere equation, which can be solved in every K\"ahler class \cite{Yau}, other equations have complicated necessary conditions that are almost as hard to verify as solving the equations themselves. To remedy this problem, numerical criteria akin to the Nakai-Moizeshon criterion for projective manifolds (and the Demailly-Paun \cite{dp} criterion in the K\"ahler case) have been conjectured \cite{gabor, collinsjacobyau} to be sufficient for the existence of smooth solutions. Results supporting such conjectures had been proven earlier under some symmetry assumptions (i.e., the toric case) \cite{collinsgabor, pindhym}, but recently, building on the work of Demailly and Paun \cite{dp}, Chen \cite{gchen} proved far more general results in the case of the J-equation and the dHYM equation. The aim of this paper is to prove similar results on projective manifolds for generalised Monge-Amp\`ere equations \cite{Pingen} - a family of PDE which includes the Monge-Amp\`ere equations, the inverse Hessian equations, and special cases of the dHYM equations.\\
\indent The setting is as follows. Let $(M,\chi)$ be a projective K\"ahler manifold and let $[\Omega_0]$ be a K\"ahler class on $M$. Let $c_k \geq 0, \ 1\leq k \leq n-1$ be constants and $f$ be a smooth function such that either
\begin{align*}
\sum_{k=1}^{n-1}c_k=0,~ f>0,
\end{align*} 
or
\begin{align*}
\sum_{k=1}^{n-1}c_k>0,~ \int_M f \chi^n \geq 0.
\end{align*} 
Assume that the following condition is met.
\begin{equation}\label{integralcondition}
\displaystyle \int_M [\Omega_0]^n = \sum_{k=1}^{n-1} \int_M c_k [\chi]^{n-k} [\Omega_0]^{k}+ \int_Mf \chi^n. 
\end{equation}
Here, and in what follows, we will suppress the wedge product to condense notation. That is, if $\al$ and $\be$ are two forms or cohomology classes, then we denote $\al\wedge\be$ by  $\al\be$. Furthermore, we assume that 
\begin{equation}\label{eqn:f}
f>f_m=- \min  \left (\frac{1}{16n}\left ( \frac{\zeta c_{\zeta}}{2n} \right )^{\frac{\zeta}{n-\zeta}} \frac{c_{\zeta} (n-\zeta)}{2n},  \frac{\zeta c_{\zeta}^2}{4n}  , \frac{c_{\zeta}^{\frac{n}{n-\zeta}}}{4n},\frac{\displaystyle \int_M [\Omega_0]^n}{4 \int_M [\chi]^n} , \frac{K}{2n}\left ( \frac{c_{\zeta}}{2 {n \choose \zeta}}\right )^{\frac{n}{n-\zeta}} \right ),
\end{equation} where  $1\leq \zeta \leq n-1$ is the largest integer such that $c_{\zeta} \neq 0,$  and $1>K>0$ is any constant less than the smallest eigenvalue of the matrix $\sum_{\vert I\vert =\zeta}E_I$ where $I$ is a multi-index and $(E_I)_{ij}=1$ if $i,j \notin I$ and $0$ otherwise. Our main result is  the  following.  
 
\begin{theorem}\label{thm:stability}
Let  $M$ be a projective connected $n$-dimensional manifold. Let $\chi$ be a fixed K\"ahler metric, and $[\Omega_0]$ be a K\"ahler class. Furthermore, let $f:M\rightarrow \RR$ be a smooth function and $\{c_k\}_{k=1}^{n-1}$ constants satisfying the  conditions above. Then the following conditions are equivalent. 
\begin{enumerate}
\item The generalised Monge-Amp\`ere equation
\begin{gather}
\Omega^n = \sum_{k=1}^{n-1} c_k \chi^{n-k} \Omega^{k}+f \chi^n,
\label{genmaequation}
\end{gather}
 has a unique smooth solution $\Omega\in[\Omega_0]$ satisfying the cone condition $$n\Omega^{n-1} - \sum_{k=1}^{n-1} c_k k \chi^{n-k} \Omega^{k-1} >0$$ 
\item (Cone condition) There exists a  K\"ahler metric $\Omega \in [\Omega_0]$ satisfying the cone condition, i.e., $$n\Omega^{n-1} - \sum_{k=1}^{n-1} c_k k \chi^{n-k} \Omega^{k-1} >0.$$
\item (Numerical condition) For all subvarieties $V\subset M$ of co-dimension $p$, we have $$\int_V\Big({n\choose p}[\Omega_0]^{n-p} - \sum_{k=p}^{n-1}c_k{k\choose p}[\chi^{n-k}]\cdot [\Omega_0^{k-p}]\Big) >0.$$ 
\end{enumerate} 
\label{thm:main2}
\end{theorem}

Note that even for the $J$-equation, this is an improvement of Gao Chen's result \cite[Theorem 1.1]{gchen}, at least for projective manifolds. Theorem \ref{thm:main2} immediately implies Conjecture 23 of Sz\'ekelyhidi \cite{gabor} (which itself generalises a conjecture of Lejmi-Sz\'ekelyhidi \cite{gaborlejmi}) for the inverse Hessian equations in the projective case. 
\begin{corollary}
Let $[\Omega_0]$ be a K\"ahler class and $\chi$ be a K\"ahler form on a projective manifold $M$. Define the constant $c:=\displaystyle \frac{\int_M [\Omega_0]^n}{\int_M [\Omega_0]^{n-k}[\chi]^k}.$ The inverse Hessian equation $$\Omega^n = c\Omega^{n-k} \wedge \chi^k$$ has a solution if and only if for every subvariety $V$ of codimension $n-k\leq p\leq n-1$, the following inequality holds.
$$ \displaystyle \int_V {n \choose p} [\Omega_0]^{n-p} - \int_V c{k \choose p} [\chi]^k [\Omega_0]^{n-p-k} >0.$$
\label{gaborcorollary}
\end{corollary}

To prove Theorem \ref{thm:stability},  it is  first necessary to obtain a version of the theorem assuming uniform positivity of intersection numbers. Namely, we  first prove the following, which is an exact extension of \cite[Theorem 1.1]{gchen} to more general inverse Hessian equations on projective manifolds.  In  fact, we can prove an equivariant version. 

\begin{theorem}\label{thm:main}
Assume that a connected compact Lie group $G$ acts on $M$ by biholomorphisms such that $M$ is strongly $G$-compatible, and $G$ preserves $\chi$ and $f$. The following conditions are equivalent.
\begin{enumerate}
\item The generalised Monge-Amp\`ere equation
\begin{gather}
\Omega^n = \sum_{k=1}^{n-1} c_k \chi^{n-k} \Omega^{k}+f \chi^n,
\label{genmaequation}
\end{gather}
 has a unique smooth $G$-invariant solution $\Omega\in[\Omega_0]$ satisfying the cone condition $$n\Omega^{n-1} - \sum_{k=1}^{n-1} c_k k \chi^{n-k} \Omega^{k-1} >0.$$ 
\item (Cone condition) There exists a $G$-invariant K\"ahler metric $\Omega \in [\Omega_0]$ satisfying the cone condition, i.e., $$n\Omega^{n-1} - \sum_{k=1}^{n-1} c_k k \chi^{n-k} \Omega^{k-1} >0.$$
\item (Uniform numerical condition) There exists a constant $\varepsilon>0$ such that for all $G$-invariant subvarieties $V\subset M$ of co-dimension $p$, we have $$\int_V\Big({n\choose p}[\Omega_0]^{n-p} - \sum_{k=p}^{n-1}c_k{k\choose p}[\chi^{n-k}]\cdot [\Omega_0^{k-p}]\Big) \geq \vep{n\choose p}\int_V[\Omega_0]^{n-p}.$$ 
\end{enumerate} 
\end{theorem}

\indent Before proceeding, we explain the $G$-compatibility. We say that $M$ is {\em $G$-compatible} if every $G$-invariant closed analytic subset contains an ample $G$-invariant divisor. Furthermore, suppose $Z$ is a $G$-invariant closed analytic subset. Consider a canonical resolution of singularities (which always exists by the results of \cite{Kollar,bieberstone} that improve on the foundational work of Hironaka on a possibly non-canonical resolution \cite{HironakaI, HironakaII}) $Z_0=Z\leftarrow Z_1 \leftarrow \ldots \leftarrow Z_r$ obtained by blowing up $M$ successively (to get $M_i$) at $G$-invariant (with respect to the lifted $G$-action on $M_i$) smooth centres $C_i \subset Z_i$ and taking the proper transform. If for any $Z$ and any such canonical resolution, $M_i$ is $G$-compatible for all $i$, then $M$ is defined to be {\em strongly $G$-compatible}. 
Note that in this paper, ``$G$-invariant divisor" simply means that the union of the irreducible subvarieties of the divisor is a $G$-invariant set. (That is, the multiplicities are not considered in this terminology.)
\begin{remark}\label{rem:toric}
\indent A projective toric manifold $M$ is strongly $(S^1)^n$-compatible. Indeed, there is an ample toric line bundle $L$ over $M$ that restricts to an ample toric line bundle over every toric subvariety $Z$. Hence, for large $k$, taking the zero locus of a section of $L^k$ that generates an invariant subspace (for the lifted action of $(S^1)^n$), we get an ample invariant divisor. Blowing up a projective toric variety at a smooth toric subvariety gives a toric variety with a toric ample line bundle $L^k \otimes [-E]$ for sufficiently large $k$. The same argument as above shows that smooth projective $T$-varieties, i.e., projective manifolds with an effective $G=(S^1)^k$ action (where $0\leq k \leq n$), are also strongly $G$-compatible.
\end{remark}

\begin{remark}
If $c_k =0 \ \forall \ k$ and $f>0$ everywhere, then the equation \eqref{genmaequation} reduces to the usual Monge-Ampere equation, and hence Theorem \ref{thm:main} follows from the proof of Calabi conjecture by \cite{Yau}, and the numerical characterization of the K\"ahler cone by  Demailly and Paun \cite{dp}. The statement $1\Leftrightarrow 2$ is known in the case $f\geq 0$ in much greater generality. (However, very few cases where $f$ is allowed to be negative appear to be studied \cite{gchen, zheng, pingthree}.) The case of the $J$-equation was studied originally in \cite{XXchen, SongWeinkove, Weinkove}. Other inverse Hessian-type equations were first studied in \cite{FangLaiMa}. \\
\indent We recover many other existing results when they are specialised to the projective case. If $G$ is taken to be trivial, $c_k=0$ for $k  = 1,2\cdots,n-2$, and $c_{n-1}>0$, then this result reduces to the main theorem in \cite{gchen}. If $G=(S^1)^n$, we recover the main theorem in \cite{collinsgabor}. Note that in the toric case,  $G$-equivariant positivity of the numerical condition and its uniform version are equivalent. More generally, by Remark \ref{rem:toric}, our theorem also applies to $T$-varieties. For certain values of the coefficients and phase angles, as explained in \cite{pindhym}, the dHYM equation is a generalised Monge-Amp\`ere equation with non-negative coefficients. Hence, for such phase angles we recover the results of \cite{gchen} when $G$ is trivial and those of \cite{pindhym} when $G=(S^1)^n$. As far as we know, our results on generalised Monge-Amp\`ere equations on projective manifolds are new in the case of $G=(S^1)^k$ (where $1\leq k\leq n-1$).
\end{remark}
\begin{remark}[A comparison with Chen's work \cite{gchen}]
The argument of  Chen in \cite{gchen} relies on the ``diagonal trick" of Demailly and Paun. This  involves solving a certain $J$-equation on the product $M\times M$ and obtaining currents that concentrate along the diagonal $\Delta\subset M\times M$.  Though our strategy of proof is inspired from \cite{gchen}, replicating the diagonal trick in the setting of more general inverse Hessian equations poses significant additional technical challenges. In particular, it  is not clear what equation must one solve on $M\times M$. The extension of our theorem to all K\"ahler manifolds is a work in progress. In this paper, we instead take a different route and use divisors as barriers, and rely on the existence of many sub-varieties on projective manifolds.  More crucially we are able to {\em  improve} Chen's result, even in the case of the $J$ equation (albeit only on projective manifolds), in that we assume only positivity of the intersection numbers, as opposed to {\em uniform positivity} in \cite[Theorem 1.1]{gchen}. This is one of the main contributions of the present work. 
\end{remark}

\begin{remark}[A comparison to the recent preprint of Jian  Song \cite{song}]
In \cite{song}, Jian Song proved a version of Theorem \ref{thm:stability} for the $J$-equation on K\"ahler manifolds using a method that parallel's Gao Chen's method in \cite{gchen}. Our attempts at relaxing the uniform numerical criteria in Theorem \ref{thm:main} grew, in part, out of some discussions with him.  Note  that the key ingredient in both  our work and  Song's work is a degenerate mass concentration  result (compare Proposition \ref{prop:conc-massnonuniform} to \cite[Theorem 5.1]{song}). It would be interesting to remove the projectivity  assumption in Theorems \ref{thm:main} and \ref{thm:stability}, but as remarked above, there are new technical difficulties that need to be  overcome. 
\end{remark}
\begin{remark}
 Using a reformulation of the dHYM equation as a generalised Monge-Amp\`ere equation, Theorem \ref{thm:main2} implies that in the three-dimensional case, when $\tan(\hth)\geq 0$, positivity of some intersection numbers is equivalent to existence of a solution. The result in \cite{gchen} proved such an assertion only under the uniform positivity assumption. Recently, \cite{tak3} proved the the most general result in this direction. 
\end{remark}
\begin{remark}
Combining Theorem \ref{thm:stability} above with Theorem 5 and Remark 1 in \cite{Zak}, it follows that the optimal lower bound for the $J$-functional is attained by some sub-variety. We refer the interested reader to \cite{Zak} for more details. \end{remark}


We end this section with outlines of the proofs of Theorem \ref{thm:main} and Theorem \ref{thm:stability}. As mentioned earlier, the proof of Theorem \ref{thm:stability} relies on the weaker Theorem \ref{thm:main}.

\subsection*{Outline of the proof of Theorem \ref{thm:main}} 
We only need to focus on the implication $3 \Rightarrow 1$ because the others are somewhat standard. In fact since $(1)\iff  (2)$ by Theorem \ref{purePDEthm}, we  actually prove that $(3)\implies (2)$. That is, given the numerical condition, we produce a $G$-invariant $\Omega\in [\Omega_0]$ satisfying the cone condition. Just as in \cite{dp, gchen}, we proceed by means of a continuity method. For $t\geq 0$, consider the equation
\begin{gather}
\Omega_{t}^n = \sum_{k=1}^{n-1} c_k \chi^{n-k} \Omega_{t}^{k}+ f\chi^n+ a_t\chi^n,
\label{cont-genmaeqnp}
\end{gather}
where $\Omega_t\in (1+t)[\Omega_0]$. Note that $a_t\xrightarrow{t\rightarrow 0^+}0$, and hence we want to  solve the above equation at $t=0$. For $t>>1$, $\hat \Omega_t =  (1+t)\Omega_0$  satisfies the cone condition   $$n\hat \Omega_{t}^{n-1} - \sum_{k=1}^{n-1} c_k k \chi^{n-k}\hat \Omega_{t}^{k-1} >0.$$ Moreover, for all $t\geq 0$, the assumptions on $f_m$ are met. By Theorem \ref{purePDEthm}, there exists a solution to \eqref{cont-genmaeqn} for $t>>1$. In particular, if we let $$I = \{t\in [0,\infty)~|~ \eqref{cont-genmaeqnp} \text{ has a solution}\},$$ then $I$ is non-empty. The infinite-dimensional implicit function theorem implies that $I$ is open. We need to show that the set is closed. By the nature of the cone condition, if $t\in I$, then $t'\in I$ for all $t'>t$. So let $t_0 = \inf I.$ It is enough to prove that $t_0\in I$. Replacing $\Omega_0$ by $(1+t_0)\Omega_0$, without loss of generality, we may assume that $t_0 = 0$ i.e.\ we have a solution to \eqref{cont-genmaeqn} for all $t>0$. 
\begin{itemize}
\item {\bf Step-1.}  Let $Y$ be a $G$-invariant ample divisor (which exists by hypothesis). By Theorem \ref{thm:conc-mass}, there is a closed, non-negative current $\Theta\in [\Omega_0]$ such that $\Theta \geq 2\beta [Y]$ for some $\beta<<1$, and $\Theta - \beta[Y]$ satisfies the cone condition (not necessarily with strict positivity though) outside $Y$ in the sense of Chen. 
\item {\bf Step-2.} For any $c>0$, the Lelong sub-level set $E_c = \{x~|~\nu(\Theta - \beta[Y],x)\geq c\}$ is a closed analytic subset that is $G$-invariant if $\Theta$ and $Y$ are $G$-invariant. Let $Z = E_c  \cup Y$, and let $Z \subset V \subset\subset U $. In a neighbourhood $U$ of $Z$, we construct a metric $\Omega_U = \Omega_0 + \ddbar\psi_U$ satisfying the cone condition using the induction hypothesis and the existence of $G$-equivariant resolutions of singularities (which follows from functoriality in \cite{bieberstone, Kollar}) (cf.\ Lemma \ref{localmetric}). 
\item {\bf Step-3.} Let $\chi_Y \in [Y]$ be a K\"ahler metric. Outside a neighbourhood (of pre-determined size) $U$ of the $G$-invariant ample divisor $Y$, we regularise $T=\Theta - \beta[Y]+\beta\chi_Y$ to obtain a metric in $[\Omega_0]$ satisfying the cone condition on $U^c$.  
\item {\bf Step-4.} Using Proposition \ref{prop:gluing} we glue the metrics obtained in Step-2 and Step-3 to get a smooth metric satisfying the cone condition. In the $G$-equivariant case, the metric obtained  can be averaged over $G$ to make it $G$-invariant. Since $G$ is connected and compact, the averaging process does not change the K\"ahler class. Moreover, since the cone condition can be written in a convex manner, the averaged metric continues to satisfy it, and we are done.
\end{itemize}

\subsection*{Outline of the proof of Theorem \ref{thm:stability}} The proof of Theorem \ref{thm:main} carries through almost verbatim, as long as one can construct a K\"ahler metric satisfying the (non-uniform) cone condition in the neighbourhood of {\em any} analytic subset $Z\subset M$ (cf. Proposition \ref{prop:nonuniformlocal}). If $Z$ is smooth, this again follows almost verbatim from the proof of Lemma \ref{lem:localconesmooth}. In general we use a resolution of singularities, a stronger mass concentration lemma for certain semi-ample classes (cf. Proposition \ref{prop:conc-massnonuniform}), and an induction on the maximal dimension of an irreducible component of $Z$. During the course of our proof, we need to choose an ample divisor $Y$ (akin to Step-1 in the proof of Theorem \ref{thm:main}) satisfying some properties. It is not clear to us if such a choice can be made in the $G$-invariant case (even under the assumption of $G$-compatibility).

\subsection*{Acknowledgements} This is work partially supported by grant F.510/25/CAS-II/2018(SAP-I) from UGC (Govt. of India). The first author (V. Datar) is grateful to the Infosys Foundation for their support through the Young Investigator Award. The second author (V. Pingali) is partially supported by a MATRICS grant MTR/2020/000100 from SERB (Govt. of India) We are deeply indebted to Gao Chen for many clarifications regarding his paper, and for his comments on all drafts of this work. We thank Jian Song for his comments on the first draft of the paper, and for many useful discussions on the issue of relaxing uniform positivity in the Theorem \ref{thm:main}. We also thank Indranil Biswas,  Apoorva Khare, Kapil Paranjape, Venkatesh Rajendran for help with regard to some algebraic aspects of the work. Finally, we thank the anonymous reviewer for detailed and useful comments.

\section{The PDE aspects}\label{purePDE}
\indent In this section we prove that the generalised Monge-Amp\`ere equation with a slight negative ``constant" term can be solved (if $\displaystyle \sum_k c_k>0$) assuming the cone condition holds. 
\begin{theorem}\label{purePDEthm}
$(1)\Leftrightarrow (2)$ in Theorem \ref{thm:main}.
\end{theorem}
\indent We prove Theorem \ref{purePDEthm} using the method of continuity. Uniqueness is standard. Given uniqueness, the equivariant version of the theorem follows easily. Hence, we may ignore $G$ for the remainder of this section. Since the case $c_k=0 \ \forall \ k, f>0$ is well-known \cite{Yau}, we assume without loss of generality that $\sum_k c_k>0$. \\
\indent Note that since $\int_M f\chi^n \geq 0$, there is a point $p\in M$ where $f(p) \geq 0$. At such a point, it is easy to see that the cone condition is met. This observation combined with the following lemma shows the necessity of the cone condition at all points on $M$.
\begin{lemma}
Suppose $f(p)\geq 0$ for some $p\in M$. Assume that Equation \ref{genmaequation} has a smooth solution $\Omega$. Then $\Omega>0$ and $n\Omega^{n-1} - \displaystyle \sum_{k=1}^{n-1} c_k \chi^{n-k} k \Omega^{k-1}>0$.
\label{necessity}
\end{lemma}
\begin{proof}
Let $q\in M$ be any other point and $\gamma(s)$ be a path connecting $p$ to $q$. Assume that $r$ is the first point on $\gamma(s)$ where either the cone condition or the K\"ahler condition becomes degenerate. It is easy to see that the cone condition must become degenerate first. We will rule out this possibility, hence showing that neither can become degenerate.\\
\indent If $f(r)\geq 0$, the result is standard. Therefore, without loss of generality assume that $f(r)<0$. Choose holomorphic normal coordinates for $\chi$ at $r$ such that $\Omega$ is diagonal with eigenvalues $\lambda_1\leq \lambda_2 \ldots$. Equation \ref{genmaequation} can be rewritten as
\begin{gather}
1=\sum_{k=1}^{n-1} c_k \frac{k! (n-k)!}{n!} S_{n-k} \left(\frac{1}{\lambda}\right) + f S_n\left(\frac{1}{\lambda}\right),
\label{eqincoord}
\end{gather}
where the symmetric polynomial $S_k(A)$ is the coefficient of $t^k$ in $\det(I+tA)$. The cone condition can be written as
\begin{gather}
1>\sum_{k=1}^{n-1} c_k \frac{k! (n-k)!}{n!} S_{n-k;i} \left(\frac{1}{\lambda}\right), 
\label{coneincoord}
\end{gather}
where $S_{n-k;i}\left(\frac{1}{\lambda}\right)$ is a symmetric polynomial in $n-1$ eigenvalues (excluding $\lambda_i$). For future use, we define $S_{n-k;i,j}=0$ when $i=j$ and the symmetric polynomial in $n-2$ eigenvalues (excluding $\lambda_i, \lambda_j$) when $i\neq j$. When it degenerates, for some $i$,
\begin{gather}
1=\sum_{k=1}^{n-1} c_k \frac{k! (n-k)!}{n!} S_{n-k;i} \left(\frac{1}{\lambda}\right).
\label{idegen}
\end{gather}
Denote $S_k \left(\frac{1}{\lambda}\right)$ by $\sigma_k$. Using $\sigma_{n-k}=\sigma_{n-k;i}+\frac{\sigma_{n-k-1;i}}{\lambda_i}$ we arrive at the following equations.
\begin{gather}
0=\displaystyle \sum_{k=1}^{n-1} c_k \frac{k! (n-k)!}{n!\lambda_i} \sigma_{n-k-1;i} +f\sigma_n \nonumber \\
\Rightarrow \sigma_{n-1;i} = \frac{1}{\vert f \vert} \displaystyle \sum_{k=1}^{n-1} c_k \frac{k! (n-k)!}{n!} \sigma_{n-k-1;i} \geq   \frac{c_{\zeta}\sigma_{n-\zeta-1;i}}{{n \choose \zeta}\vert f \vert}. 
\label{subdegensemifinal}
\end{gather}
At this juncture, we use the Maclaurin inequality to simplify the expression in \ref{subdegensemifinal}.
\begin{gather}
\frac{c_{\zeta}\sigma_{n-\zeta-1;i}}{{n \choose \zeta}\vert f \vert}\geq \frac{c_{\zeta}{n-1 \choose \zeta}\sigma_{n-1;i}^{\frac{n-\zeta-1}{n-1}}}{{n \choose \zeta}\vert f \vert}\nonumber \\
\Rightarrow \sigma_{n-1;i} \geq \left ( \frac{c_{\zeta}(n-\zeta)}{n\vert f \vert} \right )^{\frac{n-1}{\zeta}}.
\label{subdegen}
\end{gather}
From \ref{idegen} we see that
\begin{gather}
1\geq \frac{c_{\zeta}}{{n \choose \zeta}}\sigma_{n-\zeta;i} \nonumber \\
\Rightarrow \frac{{n \choose \zeta}}{c_{\zeta}} \geq \sigma_{n-\zeta;i} \geq {n-1 \choose \zeta-1} \sigma_{n-1;i}^{\frac{n-\zeta}{n-1}}.
\label{afteridgen}
\end{gather}
Comparing \ref{afteridgen} and \ref{subdegen} we see that
\begin{gather}
\vert f \vert \geq \left ( \frac{\zeta c_{\zeta}}{n} \right )^{\frac{\zeta}{n-\zeta}} \frac{c_{\zeta} (n-\zeta)}{n}.
\end{gather}
By the assumption on $f$ we arrive at a contradiction. Hence the cone condition never degenerates. From the cone condition, we can clearly see that K\"ahlerness also holds.
\end{proof}
Now we prove existence assuming that the background metric $\hat\Omega_0\in [\Omega_0]$ satisfies the cone condition. To this end, consider the following continuity path depending on a parameter $0\leq t \leq 1$.
\begin{gather}
\Omega_{t}^n = t\left (\displaystyle \sum_{k=1}^{n-1} c_k \chi^{n-k} \Omega_t^{k} +f \chi^n\right ) +(1-t)c_0\chi^n,
\label{continuitypath}
\end{gather}
where $c_0$ is a constant such that $c_0\int_M[\chi]^n=\displaystyle \int_M [\Omega_0]^n$. For a smooth function $\phi$, let $\Omega_{\phi} = \hat\Omega_0 + \spbp \phi$. We normalise $\phi$ to satisfy $\inf \phi =0$. At $t$, denote by $\phi_t$ the unique such smooth function so that $\Omega_t = \hat\Omega_0+\spbp \phi_t$.  At $t=0$, there exists a smooth solution thanks to Yau's resolution of the Calabi conjecture. For $0\leq t \leq \frac{1}{2}$, we see that $tf \chi^n +(1-t)c_0\chi^n >0$ by the assumptions on $f$. For $t>\frac{1}{2}$, we see (using the assumptions on $f$ again) that $tf +(1-t)c_0 >- \left ( \frac{t \zeta c_{\zeta}}{n } \right )^{\frac{\zeta}{n-\zeta}} \frac{tc_{\zeta} (n-\zeta)}{n} $. The proof of Lemma \ref{necessity} implies that the cone condition is met. Hence, the linearisation is invertible and the set of all such $t$ where there is a solution is open. To prove that it is closed, we need \emph{a priori} estimates. Since we are given that the cone condition is preserved along the continuity path, using Yau's Moser iteration technique (akin to \cite{Weisun} for instance), one can easily prove a $C^0$ estimate. We now prove a Laplacian estimate.
\begin{proposition}
$\Delta_{\chi}{\phi_t}\leq C $ where $C$ is independent of $t$.
\label{secondderivative}
\end{proposition}
\begin{proof}
First we prove the following lemma.
\begin{lemma}\label{usefulpropslemma}
For each $0\leq t \leq 1$, let $S_{t,p}$ be the open subset of positive-definite Hermitian matrices satisfying the cone condition for Equation \ref{continuitypath}. Cover $M$ with finitely many holomorphic charts. For a fixed point $p$, the function $F_{t,p} : S_{t,p}\rightarrow \mathbb{R}$ given by $$F_{t,p}(A)=\sum_{k=1}^{n-1} t\frac{c_k}{{n \choose k}} \sigma_{n-k} (A) + \left(tf(p) +(1-t)c_0\right)\sigma_n(A)$$ satisfies the following properties. Let $A$ have eigenvalues $\lambda_i>0$ and let $L$ be any constant larger than the maximum over all the coordinate charts of $\vert \frac{\partial f}{\partial z^{\mu}}\vert$ and $\vert \frac{\partial^2 f}{\partial z^{\mu} \partial \bar{z}^{\mu}} \vert$.   
\begin{enumerate}
\item $S_t$ is a convex open set.
\item The function $g_{t,p}(A) =F_{t,p}(A^{-1})$ extends to a smooth function (of $t,p,A$) on the orthant $\{ \lambda_i \geq 0 \}$ (and $0\leq t \leq 1, p\in M$).
\item $\frac{\partial F_{t,p}}{\partial \lambda_i}(A)<0$.
\item If $\lambda_1 \leq \lambda_2 \ldots$, then 
\begin{gather}
-\lambda_1 \frac{\partial F_{t,p}}{\partial \lambda_1}\geq - \lambda_i \frac{\partial F_{t,p}}{\partial \lambda_i} \ \forall \ i.
\label{smallesteig}
\end{gather}
\item  If $B$ is any Hermitian matrix, then 
\begin{gather}
\displaystyle \sum_{\mu,\nu,\alpha, \beta} \frac{\partial ^2 F_{t,p}}{\partial A_{\mu\nu} \partial A_{\alpha \beta}} \bar{B}_{\nu \mu} B_{\alpha \beta}+  \sum_{i,j}\frac{\partial F_{t,p}}{\partial \lambda_i} \frac{\vert B_{ij} \vert^2}{\lambda_j} -\alpha\sum_i \frac{\partial F_{t,p}}{\partial \lambda_i} \frac{\vert B_{ii} \vert^2}{\lambda_i} \geq 0.
\label{strongconvexity}
\end{gather}
where $\alpha=0$ if $2\zeta \geq n$ and $\alpha=\frac{1}{2}$ otherwise. In particular, for all $1\leq \zeta \leq n-1$, we see that $F_{t,p}$ is convex.
\item $F_{t,p}(A) >0$.
\item There is a constant $C$ independent of $t, p, A$ such that 
\begin{gather}
\frac{F_{t,p}}{C}\leq \displaystyle \sum_k -\lambda_k \frac{\partial F_{t,p}}{\partial \lambda_k} \leq CF_{t,p}.
\label{approxeuler}
\end{gather}
\item $\sum_{\mu}\vert \frac{\partial^2 F_{t,p}(A)}{\partial z^{\mu} \bar{z}^{\mu}} \vert \leq C$ where $C$ is independent of $t, p, A$.
\end{enumerate}
Suppose that there is a matrix $A_0 \in \displaystyle \cap_{0\leq t\leq 1} S_t$ with eigenvalues $0<\mu_1\leq \mu_2\ldots \leq \mu_n$. If $F_t(A,p) = 1$ for all $p\in M$, then 
\begin{enumerate}\setcounter{enumi}{8}
\item There exist constants $\theta, N>0$ independent of $t,p,A$ such that if $tr(A)>N$, then $\displaystyle \sum_i \frac{\partial F_{t,p}}{\partial \lambda_i}(A) (\lambda_i-\mu_i) \geq \theta \left (1-\sum_i \frac{\partial F_{t,p}}{\partial \lambda_i} (A) \right ).$
\item There exists a constant $C$ independent of $t, p, A$ such that for any matrix $B$ and any number $\epsilon>0$, $$\displaystyle \sum_{i,\mu} \left \vert \frac{\partial^2 F_{t,p}(A)}{\partial \lambda_i \partial z^{\mu}} B_{i\mu} \right \vert \leq \frac{C}{\epsilon} - \epsilon \displaystyle \sum_{i, \mu} \frac{\partial F_{t,p}(A)}{\partial \lambda_i} \frac{\vert B_{i\mu} \vert^2}{\lambda_{i}}. $$
\end{enumerate}
\end{lemma}
\begin{proof}
Note that the cone condition is $1>\displaystyle \sum_{k=1}^{n-1}\frac{tc_k}{{n\choose k}} \sigma_{n-k;i}$. The first few statements are proved as follows.
\begin{enumerate}
\item Since $\sigma_{n-k;i}$ is well known to be a convex function, $S_t$ is an open convex set.
\item Trivial.
\item 
\begin{gather}
-\frac{\partial F_t}{\partial \lambda_i} = \frac{1}{\lambda_i} \left ( \sum_{k=1}^{n-1}\frac{tc_k}{{n\choose k}} \frac{1}{\lambda_i}\sigma_{n-k-1;i} +\left (tf+(1-t)c_0 \right )\sigma_{n}\right ). 
\label{formulaforftderi}
\end{gather}
When $0\leq t\leq \frac{1}{2}$ the result is clear because the $\sigma_n$ term is non-negative. If $\frac{1}{2} \leq t \leq 1$, then using the Maclaurin inequality we arrive at the following inequalities.
\begin{gather}
-\lambda_i^2 \frac{\partial F_t}{\partial \lambda_i} \geq \frac{c_{\zeta}\sigma_{n-\zeta-1;i}}{2{n \choose \zeta}} -\vert f _m \vert \sigma_{n-1;i} \geq \frac{c_{\zeta}{n-1 \choose n-\zeta -1}}{2{n \choose \zeta}} \sigma_{n-1;i}^{\frac{{n-2 \choose n-\zeta -2}}{{n-1 \choose n-\zeta -1}}}-\vert f _m \vert \sigma_{n-1;i} \nonumber \\
= \sigma_{n-1;i} ^{\frac{n-\zeta-1}{n-1}} \left ( \frac{c_{\zeta}(n-\zeta)}{2n} -\vert f _m \vert \sigma_{n-1;i} ^{\frac{\zeta}{n-1}} \right ),
\label{formulaforftderi2}
\end{gather}
which is seen to be positive using the cone condition and the assumptions on $f_m$.
\item Note that for $k<n$, $-\lambda_1 \partial_1 \sigma_k \geq -\lambda_i \partial_i \sigma_k$ and for $k=n$, $-\lambda_1 \partial_1 \sigma_n = -\lambda_i \partial_i \sigma_n$. Hence, the result follows.
\item
  From the proof of Lemma 8 in \cite{collinsgabor}, we see that 
\begin{gather}
\displaystyle \sum_{\mu,\nu,\alpha, \beta} \frac{\partial ^2 F_{t,p}}{\partial A_{\mu\nu} \partial A_{\alpha \beta}} \bar{B}_{\nu \mu} B_{\alpha \beta}+ \sum_{i,j}\frac{\partial F_{t,p}}{\partial \lambda_i} \frac{\vert B_{ij} \vert^2}{\lambda_j} \nonumber \\
\geq \sum \frac{tc_k}{{n \choose k}}\frac{B_{\alpha \alpha} \bar{B}_{\mu \mu}}{\lambda_{\alpha} \lambda_{\mu}}\frac{S_{k;\alpha,\mu}+\delta_{\alpha \mu}S_{k;\alpha}}{S_n} +\frac{tf+(1-t)c_0}{S_n}\frac{ B_{\alpha \alpha} \bar{B}_{\mu \mu}}{\lambda_{\alpha}\lambda_{\mu}}.
\label{justbeforeM}
\end{gather}
The matrix $M_{ij} = \sum \frac{tc_k}{{n \choose k}}\left (S_{k;ij}+\delta_{ij}S_{k;i} \right )+ \left (tf+(1-t)c_0 \right)$ can be written as (akin to \cite{FangLaiMa, collinsgabor}) $M=\sum \frac{tc_k}{{n \choose k}}\sum_{\Vert I \Vert =k}\lambda_I E_I  + \left ( tf+(1-t)c_0 \right ) \mathbf{1}$ where for any $k$-tuple $I$, the positive semidefinite matrix $E_I$ is defined to have entries $(E_I)_{ij}=1$ if $i,j \notin I$ and $0$ otherwise; moreover, $\lambda_I=\lambda_{i_1} \lambda_{i_2}\ldots\lambda_{i_k}$, and the matrix $\mathbf{1}$ is defined as $(\mathbf{1})_{ij}=1 \ \forall \ i,j$. \\
If $0\leq t\leq \frac{1}{2}$, it is easy to see by the assumptions on $f$ that $M$ is non-negative (because the constant term is so). Assume $t>\frac{1}{2}$. For future use, we observe that $0 \leq \mathbf{1} \leq nId$. Now we have two cases.
\begin{enumerate}
\item $2\zeta \geq n$ :Using the cone condition and the assumption that $2\zeta \geq n$, we see that $\lambda_I >\left (\frac{c_{\zeta}}{2{n \choose \zeta}}\right )^{\frac{\zeta}{n-\zeta}} \ \forall $ multi-indices $I$ with $\vert I \vert = \zeta$. Recall that $K>0$ is the least eigenvalue of $\sum_{\Vert I\Vert = \zeta} E_I$. Now note that $M>t\left (\frac{c_{\zeta}}{2{n \choose \zeta}} \right )^{\frac{n}{n-\zeta}}KI-t\vert f_m \vert nI$, which is positive-definite using the assumptions on $f_m$.
\item $1\leq \zeta \leq n-1$ : Note that
\begin{gather}
\displaystyle \sum_{\mu,\nu,\alpha, \beta} \frac{\partial ^2 F_{t,p}}{\partial A_{\mu\nu} \partial A_{\alpha \beta}} \bar{B}_{\nu \mu} B_{\alpha \beta}+ \sum_{i,j}\frac{\partial F_{t,p}}{\partial \lambda_i} \frac{\vert B_{ij} \vert^2}{\lambda_j}  - \frac{1}{2}\sum_{i}\frac{\partial F_{t,p}}{\partial \lambda_i} \frac{\vert B_{ii} \vert^2}{\lambda_i} \nonumber \\
\geq \displaystyle \sum_{i,j} \sigma_n M_{ij}\frac{B_{ii}B_{jj}}{\lambda_i \lambda_j} -\frac{1}{2}\sum_i \frac{\partial F_{t,p}}{\partial \lambda_i} \frac{\vert B_{ii} \vert^2}{\lambda_i}.
\label{tempconvex}
\end{gather} 
Using the first inequality in \ref{formulaforftderi2} we see that
\begin{gather}
\displaystyle \sum_{\mu,\nu,\alpha, \beta} \frac{\partial ^2 F_{t,p}}{\partial A_{\mu\nu} \partial A_{\alpha \beta}} \bar{B}_{\nu \mu} B_{\alpha \beta}+ \sum_{i,j}\frac{\partial F_{t,p}}{\partial \lambda_i} \frac{\vert B_{ij} \vert^2}{\lambda_j} - \frac{1}{2}\sum_{i}\frac{\partial F_{t,p}}{\partial \lambda_i} \frac{\vert B_{ii} \vert^2}{\lambda_i} \nonumber \\
\geq 2\left(-\sum_i n\sigma_{n-1;i} \vert f_m \vert \frac{\vert B_{ii} \vert^2}{\lambda_i^3} + \frac{1}{8} \sum_i  \frac{\vert B_{ii} \vert^2}{\lambda_i^3}\frac{c_{\zeta}\sigma_{n-\zeta-1;i}}{{n \choose \zeta}}\right).
\end{gather}
As before, the Newton-Maclaurin inequalities and the assumptions on $\vert f_m \vert$ imply the desired result.

\end{enumerate}
\item If $0\leq t\leq \frac{1}{2}$, it is clear (as before) that $F_t(A)>0$. If $t>\frac{1}{2}$, then $$F_t(A)\geq \frac{c_{\zeta}}{2}\sigma_{n}^{\frac{n-\zeta}{n}}-\vert f_m \vert\sigma_n=\sigma_n^{\frac{n-\zeta}{n}}\left (\frac{c_{\zeta}}{2}-\vert f_m \vert \sigma_n^{\frac{\zeta}{n}}  \right).$$  At this point we see that the cone condition implies that $1>\frac{1}{2}\frac{c_{\zeta}\sigma_{n-\zeta;i}}{{n \choose \zeta}}$ because $t>\frac{1}{2}$. Using the Maclaurin inequality we see that 
\begin{gather}
\frac{2n}{\zeta c_{\zeta}}>\sigma_{n-1;i}^{\frac{n-\zeta}{n-1}} \nonumber \\
\Rightarrow \left(\frac{2n}{\zeta c_{\zeta}}\right)^n>\Pi_{i=1}^n \sigma_{n-1;i}^{\frac{n-\zeta}{n-1}} =\sigma_n^{n-\zeta}.
\end{gather}
Therefore, by the assumptions on $f$, we see that $F_t(A)>0$. 
\item Noting that $\displaystyle \sum_i -\lambda_i \frac{\partial F_t}{\partial \lambda_i} = \sum_k \frac{t(n-k)c_k}{{n \choose k}} \sigma_{n-k}+n \left(tf+(1-t)c_0 \right) \sigma_n$, it is easy to see using the assumptions on $f$ that the result holds. 
\item Clearly, $\vert \frac{\partial^2 F_{t,p}}{\partial z^{\mu} \partial \bar{z}^{\mu}} \vert \leq L \sigma_n(A) \leq C$.
\end{enumerate}
Now we assume that $F_t(A)=1$. 
\begin{enumerate}\setcounter{enumi}{8}
\item Firstly, we note that $N<\mathrm{tr}(A)=\displaystyle \sum_{i=1}^n \lambda_i \leq n\lambda_n$. As can be seen from \cite{collinsgabor}, the above properties imply the result. Indeed, suppose the result is false. Noting that $\theta$ is allowed to depend on $\mu_i$, for a given $\mu_i$ there is a sequence $\lambda_{i,N}, t_N, p_N$ (where $N\rightarrow \infty$) such that the following inequality holds. (We drop the dependence on $N$ for ease of notation.) 
\begin{gather}
\displaystyle \sum_i \frac{\partial F_{t,p}}{\partial \lambda_i}(\lambda) (\lambda_i-\mu_i) < \frac{1}{N} \left (1-\sum_i \frac{\partial F_{t,p}}{\partial \lambda_i} (A) \right ).
\label{ineq:collinsgaborcontradictoryassumption}
\end{gather} 
Since $\frac{\partial F_{t,p}}{\partial \lambda_i}(\lambda) \lambda_i \geq -CF(A)=-C$, 
\begin{gather}
\displaystyle -C-\frac{1}{N}+\sum_i \left(-\frac{\partial F_{t,p}}{\partial \lambda_i}(\lambda) \right)\left(\mu_i-\frac{1}{N} \right)<0.
\end{gather}
If $-\frac{\partial F_{t,p}}{\partial \lambda_1}(\lambda)$ is unbounded (from above) then for sufficiently large $N$, we have a contradiction. Thus $-\frac{\partial F_{t,p}}{\partial \lambda_1}(\lambda)>0$ is bounded from above and since $-n\lambda_1 \frac{\partial F_{t,p}}{\partial \lambda_1}(\lambda)\geq \displaystyle -\sum_i \lambda_i \frac{\partial F_{t,p}}{\partial \lambda_i}(\lambda) \geq \frac{1}{C}$, we see that $\lambda_1>\frac{1}{C'}$ for some $C'$ independent of $t,p,A$. Therefore, the function $g_{t,p}$ is smoothly controlled (independent of $N$). \\
\indent As in \cite{collinsgabor}, denote by $\tilde{F}_{t,p}(\lambda)$ the smooth convex function $\displaystyle \lim_{\lambda_n\rightarrow \infty} F_{t,p}(\lambda)=g_{t,p}(\lambda_1,\ldots,\lambda_{n-1},0)$. By convexity, 
\begin{gather}
\tilde{F}_{t,p}(\mu)\geq \tilde{F}_{t,p}(\lambda)+\displaystyle \sum_{i=1}^{n-1}\frac{\partial \tilde{F}_{t,p}}{\partial \lambda_i}(\lambda) (\mu_i-\lambda_i).
\label{ineq:convexityforsmaller}
\end{gather}
Therefore,
\begin{gather}
\displaystyle \sum_i \frac{\partial F_{t,p}}{\partial \lambda_i}(\lambda) (\lambda_i-\mu_i) \geq \displaystyle \sum_{i=1}^{n} \left(\frac{\partial F_{t,p}}{\partial \lambda_i}(\lambda)- \frac{\partial \tilde{F}_{t,p}}{\partial \lambda_i}(\lambda)\right)(\lambda_i-\mu_i)+\tilde{F}_{t,p}(\lambda)-\tilde{F}_{t,p}(\mu).
\label{ineq:afterconvexityforsmaller}
\end{gather}
Since $F_{t,p}(\lambda)=1$ for sufficiently large $N$, $\tilde{F}_{t,p}(\lambda_i)> 1-\tau$ for a small $\tau$. Since $\mu$ satisfies the cone condition, $\tilde{F}_{t,p}(\mu)\leq 1-\delta$ for some $\delta>0$. It is easy to see that
\begin{gather}
\frac{\partial F_{t,p}}{\partial \lambda_i}(\lambda)- \frac{\partial \tilde{F}_{t,p}}{\partial \lambda_i}(\lambda)=-x_i^2 \left(\frac{\partial g_{t,p}(x_1,\ldots,x_n)}{\partial x_i}- \frac{\partial g_{t,p}(x_1,\ldots,0)}{\partial x_i}\right)\nonumber \\
=-x_i^2 x_n h_{t,p}(x),
\label{eq:intermsofg}
\end{gather}
where $x_i=\frac{1}{\lambda_i}$, and $h_{t,p}$ is smooth and uniformly controlled independent of $N$. Therefore, for sufficiently large $N$ (and small $\tau$),
\begin{gather}
\displaystyle \sum_i \frac{\partial F_{t,p}}{\partial \lambda_i}(\lambda) (\lambda_i-\mu_i) \geq \frac{\delta}{2}.
\label{ineq:forlargeN}
\end{gather}
Since $\frac{\partial F}{\partial \lambda_i}$ is bounded, we arrive at a contradiction between Inequalities \ref{ineq:forlargeN} and \ref{ineq:collinsgaborcontradictoryassumption} for large $N$.
\item Note that $\vert \frac{\partial^2 F_{t,p}(A)}{\partial \lambda_i \partial z^{\mu}} \vert \leq  \frac{L \sigma_n}{\lambda_i} $. Akin to the above properties, one can prove that $-\lambda_i^2 \frac{\partial F}{\partial \lambda_i} \geq \frac{\sigma_{n-1;i}}{C'}$ for some $C'$. Now, 
\begin{align}
\displaystyle \sum_i \frac{L\sigma_n \vert B_{i\mu} \vert }{\lambda_i} &\leq -\displaystyle \sum_{i,\mu}LC'\frac{\partial F_{t,p}}{\partial \lambda_i} \vert B_{i\mu} \vert \nonumber \\
&\leq -\frac{C''}{\epsilon} \sum_i \frac{\partial F_{t,p}}{\partial \lambda_i} \lambda_i -\epsilon \sum_{i,\mu} \frac{\partial F_{t,p}}{\partial \lambda_i} \frac{\vert B_{i\mu} \vert^2}{\lambda_i} \nonumber \\
&\leq \frac{C}{\epsilon} -\epsilon \sum_{i,\mu} \frac{\partial F_{t,p}}{\partial \lambda_i} \frac{\vert B_{i\mu} \vert^2}{\lambda_i}
\end{align}
\end{enumerate}
\end{proof}
 The aforementioned lemma is enough to guarantee that Proposition \ref{secondderivative} holds when $2\zeta \geq n$ by the method of \cite{pindhym}, when followed word-to-word. In fact, such an estimate works for any function $F_{t,p}$ that satisfies the above conditions (along with property \ref{strongconvexity} for $\alpha=0$). However, when $\alpha=\frac{1}{2}$, we need to do more work. To this end, we prove a general proposition about estimates.
\begin{proposition}
For any family of PDE of the form $F_{t,p} (A)=1$, where the endomorphism $A$ is given by $A_{k}^i :=\chi^{i\bar j}(\Omega_{\phi_t})_{k\bar j}$, the solution $\phi_t$ (normalised to satisfy $\inf \phi_t =0$) satisfies the  \emph{a priori} estimate 
\begin{gather}
 \Delta_{\chi} \phi_t \leq C (1+\vert \nabla_{\chi} \phi_t \vert^2),
\label{c2estimate}
\end{gather}
if $\Vert \phi_t \Vert_{C^0}$ is assumed to be uniformly bounded, all the properties in Lemma \ref{usefulpropslemma} are satisfied by a symmetric function $F_{t,p}$ whose domain $S_{t,p}$ is an open convex cone lying in the set of Hermitian positive-definite matrices, and there is a K\"ahler metric $\Omega$ in the class $[\Omega_0]$ such that $\chi(p)^{i\bar j}\Omega(p)_{k\bar j}\in \displaystyle \cap_{0\leq t \leq 1} S_{t,p}$ for all $p\in M$. 
\label{c2estimateprop}
\end{proposition}
\begin{proof}
 The following calculations are inspired from similar ones in \cite{collinsjacobyau, gabor}. \\
\indent For a function (akin to \cite{phsturm}) $\gamma : [0,  \infty) \rightarrow \mathbb{R}$ defined as $$ \gamma (x) = Ax - \frac{1}{1+x}+1,$$ where $A>1$ is a constant to be chosen later, the following properties hold. We suppress the $t$ subscript for $\phi_t$ in this proof.
\begin{gather}
Ax \leq \gamma (x) \leq Ax+1 \nonumber \\
A \leq \gamma'(x) \leq 2A \nonumber \\
\gamma''(x)=-\frac{2}{(1+x)^3} \leq 0.
\label{propsofgamma}
\end{gather}
Let $G(\Omega_{\phi})=-\gamma(\phi) + \ln (\lambda_n)$ where $\lambda_1 \leq \lambda_2 \leq \ldots \lambda_n$ are the eigenvalues of $A$. Note that $G$ is a continuous function. Suppose the maximum of $G$ is attained at a point $p$. Choose holomorphic normal coordinates for $\chi$ at $p$ such that $\Omega_{\phi}$ is diagonal. In general, $\lambda_n$ is not smooth near $p$. The standard remedy is to consider a small constant diagonal matrix $B$  such that $0=B_{nn}<B_{n-1 n-1}<\ldots<B_{11}$ and a new matrix-valued function $\tilde{A}=A-B$. $B$ is chosen to be small enough so that $\tilde{A}$ is still positive-definite and lies in $S_t$ for all points $q$ in a neighbourhood $U$ of $p$. The eigenvalues of $\tilde{A}$ are distinct and are hence smooth in $U$. The function $\tilde{G} = G(\tilde{A})$ continues to achieve a local maximum at $p$. In what follows, the notation $H_{,i}, H_{,i\bar{j}}$ denotes partial derivatives (as opposed to covariant derivatives). At $p$, we differentiate with respect to $z^i, \bar{z}^i$ to obtain the following. 
\begin{gather}
0=\tilde{G}_{,i} (\tilde{A}) = -\gamma' \phi_{,i}+\frac{1}{\lambda_n(\tilde{A})}\frac{\partial \lambda_n}{\partial \tilde{A}_{\mu \nu}}\tilde{A}_{\mu \nu, i}
\nonumber \\
 = -\gamma' \phi_{,i} +\frac{(\Omega_{\phi})_{n \bar{n}, i}}{\lambda_n(A)}
\label{firstofG}
\end{gather} 
We drop the $t,p$ subscripts from now onwards. We differentiate again to obtain the following.
\begin{align}
0&\geq \displaystyle -\sum_i \frac{\partial F}{\partial \lambda_i}(A) \tilde{G}_{,i\bar{i}} (\tilde{A})=\sum_i \left(- \frac{\partial F}{\partial \lambda_i}(A) \right ) \Bigg (-\gamma'' \vert \phi_{,i} \vert^2 -\gamma' \phi_{,i\bar{i}}-\frac{\vert (\Omega_{\phi})_{n\bar{n},i} \vert^2}{\lambda_n^2(A)} \nonumber \\
&+\frac{1}{\lambda_n(A)}\frac{\partial^2 \lambda_n(\tilde{A})}{\partial \tilde{A}_{p q} \partial \tilde{A}_{r s}}(\Omega_{\phi})_{p q, i} (\Omega_{\phi})_{r s, \bar{i}} + \frac{(\Omega_{\phi})_{n\bar{n}, i\bar{i}} + \lambda_n (A) R_{n\bar{n} i\bar{i}}}{\lambda_n(A)} \Bigg ).
\label{secondofG}
\end{align}
From now onwards, we denote $\lambda_i = \lambda_i(A)$ and $\tilde{\lambda}_i = \lambda_i(\tilde{A})= \lambda_i-B_{ii}$. Assume that $\lambda_n > N$ in order to use property $9$ of Lemma \ref{usefulpropslemma}. Assume that $\vert R_{i\bar{i} j\bar{j}} \vert\leq C_1 \ \forall \ i,j$.  Using the formulae from \cite{gabor} we see that
\begin{align}
0&\geq \sum_i \left(-\frac{\partial F}{\partial \lambda_i} \right ) \Bigg (-\gamma'' \vert \phi_{,i} \vert^2 -\frac{\vert (\Omega_{\phi})_{n\bar{n},i} \vert^2}{\lambda_n^2} + \frac{(\Omega_{\phi})_{n\bar{n}, i\bar{i}} }{\lambda_n} \nonumber \\
&+\frac{1}{\lambda_n} \sum_{p<n}\frac{\vert(\Omega_{\phi})_{p\bar{n},i}\vert^2 + \vert (\Omega_{\phi})_{n\bar{p},i} \vert^2}{\lambda_n-\tilde{\lambda}_p}  \Bigg ) + \frac{\gamma' \theta}{2} \left (1- \sum_i \frac{\partial F}{\partial \lambda_i}\right ) 
\label{derivofGbeforePDE}
\end{align}
where we assumed that $\frac{\gamma' \theta}{2} \geq \frac{A\theta}{2} \geq C_1$. Now we compute the derivatives of the PDE $F(A,p)=1$. 
\begin{align}
F_{,n} + \frac{\partial F}{\partial A_{pq}}(A_{pq})_{,n}&=0 \nonumber \\
\Rightarrow F_{,n} + \displaystyle \sum_i \frac{\partial F}{\partial \lambda_i} (\Omega_{\phi})_{i\bar{i},n}&=0.
\label{firstderiPDE}
\end{align}
We differentiate again to obtain the following equality.
\begin{align}
0& =F_{,n\bar{n}}+\displaystyle \sum_i2Re(\frac{\partial F_{,n}}{\partial \lambda_i}(\Omega_{\phi})_{i\bar{i},\bar{n}}) + \sum_{p,q,r,s} \frac{\partial^2 F}{\partial A_{pq} \partial A_{rs}}(\Omega_{\phi})_{p\bar{q},n} (\Omega_{\phi})_{r\bar{s},n}\nonumber \\
 &+\sum_i \frac{\partial F}{\partial \lambda_i}((\Omega_{\phi})_{i\bar{i},n\bar{n}}+R_{i\bar{i}n\bar{n}}\lambda_i) 
\label{secondderiPDE}
\end{align}
We substitute the fourth order term from Equation \ref{secondderiPDE} into \ref{derivofGbeforePDE}, and use the assumptions of Lemma \ref{usefulpropslemma} to get the following.
\begin{align}
0&\geq \frac{\gamma' \theta}{2} \left (1- \sum_i \frac{\partial F}{\partial \lambda_i}\right )  + \displaystyle \sum_i \frac{\partial F}{\partial \lambda_i} \left (\gamma'' \vert \phi_{,i} \vert^2+ \frac{\vert (\Omega_{\phi})_{n\bar{n},i} \vert^2}{\lambda_n^2} \right )+\frac{F_{,n\bar{n}}}{\lambda_n} \nonumber \\
&+\frac{1}{\lambda_n} \Bigg (C_1 \sum_i \lambda_i \frac{\partial F}{\partial \lambda_i} -C+\frac{1}{2}\sum_i \frac{\partial F}{\partial \lambda_i} \frac{\vert(\Omega_{\phi})_{i\bar{i},n}\vert^2}{\lambda_i}-\sum_{i,j}\frac{\partial F}{\partial \lambda_i} \frac{\vert (\Omega_{\phi})_{i\bar{j},n} \vert^2}{\lambda_j} \nonumber \\
&+\frac{1}{2}\sum_i \frac{\partial F}{\partial \lambda_i} \frac{\vert (\Omega_{\phi})_{i\bar{i},n} \vert^2}{\lambda_i} \Bigg ) \nonumber \\
&\geq \frac{\gamma' \theta}{2} \left (1- \sum_i \frac{\partial F}{\partial \lambda_i}\right ) - \frac{C}{\lambda_n}+\frac{\partial F}{\partial \lambda_n}\frac{\vert (\Omega_{\phi})_{n\bar{n},n}\vert^2}{\lambda_n^2}.
\label{almostthere}
\end{align}
At this juncture, we use Equation \ref{firstofG} in \ref{almostthere} to get the following inequality.
\begin{align}
0&\geq \frac{\gamma' \theta}{2}  - \frac{C}{\lambda_n}+\lambda_n\frac{\partial F}{\partial \lambda_n}\frac{(\gamma')^2 \vert \phi_{,n} \vert^2}{\lambda_n}.
\label{almostalmost}
\end{align}
Assume that $\lambda_n >1$ and that $A\frac{\theta}{4} \geq C$. Note that $-C\leq \lambda_n \frac{\partial F}{\partial \lambda_n} \leq 0$ since $$\displaystyle -C \leq \sum_i \lambda_i \frac{\partial F}{\partial \lambda_i} $$ by assumption. Thus,
\begin{align}
0&\geq \frac{A\theta}{4}-\frac{4CA^2\vert \phi_{,n} \vert^2}{\lambda_n}.
\end{align}
Therefore, $\lambda_n \leq C(1+\vert \nabla \phi \vert^2)$.
\end{proof}
\indent Using Proposition \ref{c2estimateprop} and Lemma \ref{usefulpropslemma}, we see that in our case, indeed $$\Delta \phi \leq C(1+\vert \nabla \phi \vert^2).$$ Now, Proposition 5.1 of \cite{collinsjacobyau} (which is based on a blow-up argument) shows that $\Delta \phi \leq C$, as desired. 
\end{proof}
\indent At this point, we note that since $F_t(A)$ is convex on a convex open set, the complex version of Evans-Krylov theory akin to \cite{Siu} would have implied the $C^{2,\alpha}$ \emph{a priori} estimates on $\phi_t$ if $f$ was constant. In our case, despite $f$ not being a constant, we can still get $C^{2,\alpha}$ estimates on $\phi_t$  using a blow-up argument (like in the proof of Lemma 6.2 in \cite{collinsjacobyau} for instance). Using Schauder theory, we can get $C^{k,\alpha}$ bounds (depending on $k$) for every $k$. Applying the Arzela-Ascoli theorem and choosing a diagonal subsequence, we get closedness of the set of $t \in [0,1]$ such that Equation \ref{continuitypath} has a smooth solution. Therefore, we have a smooth solution at $t=1$. \qed
\section{Concentration of mass}\label{concentration}

We begin with a definition of what it means for a possibly singular $(1,1)$ current to satisfy the cone condition. Let $\rho:[0,1]\rightarrow \RR$ be a smooth function such that $$|\mathbb{S}^{2n-1}|\int_0^1\rho(t)t^{2n-1}dt = 1,$$ where $|\mathbb{S}^{2n-1}|$ is the Lebesgue measure of the unit sphere in $\CC^n$. For any $1>\delta>0$, and any $L^1$-function $\varphi:B_1(0)\rightarrow\RR$ on the unit ball $B_1(0)$, the {\em $\delta$-mollification} is defined to be $$\varphi_\delta(x) = \delta^{-2n}\int_{B_1(0)}\varphi(x-y)\rho\Big(\frac{|y|}{\delta}\Big)dy.$$ Then we have the following  definition from \cite{gchen}.

\begin{definition}\label{degenconedef}
Let $\Theta$ be a closed, positive $(1,1)$ current. We say that $\Theta$ satisfies the \emph{($\vep$-)uniform cone condition} 
\begin{equation}\label{singcone}
n(1-\vep)\Theta^{n-1} - \sum_{k=1}^{n-1} c_k k \chi^{n-k} \Theta^{k-1} \geq 0,
\end{equation}
if for any coordinate chart $U$ with $\Theta\Big|_U = \ddbar\varphi_U$, on $U_\delta:=\{x\in M~|~ B(x,\delta)\subset U\}$ (where $B(x,\delta)$ is a coordinate Euclidean ball of radius $\delta$ centred at $x$), we have $$n(1-\vep)(\ddbar\varphi_{U,\delta})^{n-1} (x) - \sum_{k=1}^{n-1} c_k k \chi_0^{n-k} (\ddbar\varphi_{U,\delta})^{k-1} (x) \geq 0,$$  for any K\"ahler metric $\chi_0$ on $U$ with constant coefficients satisfying $\chi_0\leq \chi$ on $B(x,\delta)$. Here $\varphi_{U,\delta}$ are the mollifications of $\varphi_U$ as above. \\
\indent If there exist strictly positive sequences $\vep_i, \mu_i \rightarrow 0$ such that $\Theta+\mu_i \chi$ satisfies the ($\vep_i$-) uniform cone condition, then we see that $\Theta$ satisfies the \emph{degenerate cone condition}. 
\end{definition}
\begin{remark}
If $\Theta$ is a smooth form, clearly the definitions above coincide with the usual pointwise definition. In particular, the degenerate cone condition boils down to the pointwise condition $n\Theta^{n-1} - \sum_{k=1}^{n-1} c_k k \chi^{n-k} \Theta^{k-1} \geq 0$.
\end{remark}
 The main goal of this section is to prove the following ``mass concentration" result. 
 \begin{theorem}\label{thm:conc-mass}
  Let $(M,\chi)$ be a connected K\"ahler manifold and $\Omega_0$ another K\"ahler metric. Suppose for all $t>0$, there exists a $\hat\Omega_t \in (1+t)[\Omega_0]$ such that $$n\hat \Omega_{t}^{n-1} - \sum_{k=1}^{n-1} c_k k \chi^{n-k} \hat\Omega_{t}^{k-1} >0.$$ Then for any divisor $Y$, there exists a $\beta_Y>0$ and a current $\Theta \in [\Omega_0]$ such that $\Theta\geq \be_Y [Y]$ and $\Theta$ satisfies the degenerate cone condition  in the sense above. If $G$ is a connected compact Lie group acting on $M$ such that $\chi$ and $Y$ are $G$-invariant, then $\Theta$ can be chosen to be $G$-invariant as well.
 \end{theorem}
 
 \begin{proof} The  proof is essentially a combination of the ideas in \cite{gchen} and \cite{dp}.  Suppose that $M$ be covered by coordinate balls $\{B_j\}_{j=1}^N$, and that $Y$ is given by the vanishing of coordinate functions $\{g_{j}\}$ on $B_j$.  Let $\{\theta_j\}$ be the partition of unity subordinate to $B_j$. For $0\leq t<<1$, define $$\psi_{t} = \log\Big(\sum_{j}(\theta_j\sum_{k}|g_{j,k}|^2) + t^2\Big),~\chi_{t} = \chi+A^{-1}\ddbar\psi_{t},$$ where we pick a suitable $A>>1$ below.  For $t$ small enough, $\chi_{t}$ is a smooth K\"ahler form on $M$ which concentrates near $Y$,  and $\chi_{t} >  (1-C A^{-1})\chi$  for some fixed $C$ (independent of $t$).  Now consider the following equation:\
 \begin{gather}
\Omega_{t}^n = \sum_{k=1}^{n-1} c_k \chi^{n-k} \Omega_{t}^{k}+f_{t}\chi^n,
\label{pertgenmaeqn}
\end{gather}
where $\Omega_{t} \in (1+t)[\Omega_0]$ and $$f_{t} = \frac{\chi_{t}^n}{\chi^n} - 1 + A_t,$$ and the constant $A_t$ is chosen so that $$\displaystyle \int_M \Omega_{t} ^n = \sum_{k=1}^{n-1} \int_M c_k \chi^{n-k} \Omega_{t}^{k}+ \int_M f_{t} \chi^n.$$ Note that $A_t>0$ for all $t$. Also,  for $\vep<<1$, we have that $f_{t} > f_m$, and hence by Theorem \ref{purePDEthm}, there is a solution  $\Omega_{t}\in (1+t)[\Omega_0]$ to \eqref{pertgenmaeqn} satisfying the cone condition  $$n\Omega_{t}^{n-1} - \sum_{k=1}^{n-1} c_k k \chi^{n-k} \Omega_{t}^{k-1} >0.$$ Now the family of currents $\Omega_{t}$ is bounded in mass since $$\int_M\Omega_{t}\wedge \chi^{n-1} = (1+t)[\Omega_0]\cdot[\chi^{n-1}]<C.$$ Define $\tilde{\Omega}_t$ as the average of $\Omega_t$ over $G$. Since the cohomology class of $\Omega_t$ does not change by averaging over a compact connected Lie group, the family $\tilde{\Omega}_t$ is also bounded in mass.  Hence, there exists a sequence $t_i\rightarrow 0^+$ such that $\tilde{\Omega}_{t_i,\vep}\rightarrow \Theta$, where $\Theta$ is a non-negative current in the class $\al$. Clearly, $\Theta$ is $G$-invariant. We now show that $\Theta \geq \be_Y[Y]$ for some  $\be_Y>0$ following the line of argument in \cite{dp}. First  we have the following simple observation. 
\begin{lemma}\label{pos-mass}
For any neighbourhood $U$ of a point $y\in Y$, there exist constants $\delta_U,t_U>0$ such that for all $t<t_U$ the following inequality holds. $$\int_{U\cap V_t}\Omega_{t}\wedge\chi^{n-1} > \delta_U,$$ where $V_t = \{\psi_0<\log t\}.$
\end{lemma}
\begin{proof}  Let $0\leq \la_{1}(z)\leq \cdots\leq \la_n(z)$ be the eigenvalues  of $\Omega_{t}$ with respect to $\chi_t$. Since $\Omega_t^{n-1}\wedge \chi_t\geq \la_2\cdots\la_n \chi_t^n$, we have that $$\int_{M} \la_2\cdots\la_n \chi_t^n \leq \int_M\Omega_t^{n-1}\wedge \chi_t = (1+t)^{n-1}[\Omega_0^{n-1}]\cdot[\chi] \leq C.$$ In particular, for any $\delta>0$, if $E_\delta := \{z\in M~|~\la_2(x)\cdots\la_n(x) > C\delta^{-1}\} $, then $$\int_{E_\delta}\chi_t^n\leq \delta.$$ By Lemma 2.1 in \cite{dp}, since $\chi< 2\chi_t$, we have that $$\int_{E_\delta}\chi_t\wedge\chi^{n-1} < 2^{n-1}\int_{E_\delta}\chi_t^n< 2^{n-1}\delta.$$ Next, by \cite[Lemma 2.1(iii)]{dp}, there exists a $\delta(U)$ such that for all $t<<1$, $$\int_{U\cap V_t\cap E_{\delta}^c}\chi_t\wedge \chi^{n-1} = \int_{U_\cap V_t}\chi_t\wedge\chi^{n-1} - \int_{U\cap V_t\cap E_\delta}\chi_t\wedge\chi^{n-1} \geq \delta(U) -2^{n-1}\delta = \frac{3}{4}\delta(U),$$ where we let $\delta := \frac{\delta(U)}{2^{n+1}}$.
On the other hand, from our equation, $\Omega_t^n \geq f_t\chi^n \geq \chi_t^n - \chi^n.$ So if $z\in E_\delta^c$, then  $$\Omega_t(z) \geq \frac{\delta}{C}\left (1-\frac{\chi^n(z)}{\chi_t^n(z)} \right)\chi_t(z).$$  Integrating, 
\begin{align*}
\int_{U\cap V_t}\Omega_{t}\wedge\chi^{n-1} \geq \int_{U\cap V_t\cap E_\delta^c}\Omega_{t}\wedge\chi^{n-1}&\geq \frac{\delta}{C}\Big(\int_{U\cap V_t\cap E_\delta^c}\chi_t\wedge\chi^{n-1}  -  \int_{U\cap V_t\cap E_\delta^c}\frac{\chi^n}{\chi_t^n}\chi_t\wedge \chi^{n-1}\Big) \\
&\geq \frac{\delta}{C}\Big(\frac{3\delta(U)}{4} -    \int_{U\cap V_t\cap E_\delta^c}\frac{\chi^n}{\chi_t^n}\chi_t\wedge \chi^{n-1}\Big).
\end{align*}
For the second term we estimate $$ \int_{U\cap V_t\cap E_\delta^c}\frac{\chi^n}{\chi_t^n}\chi_t\wedge \chi^{n-1} < 2^{n-1}\int_{U\cap V_t}\chi^n,$$ and hence it can be made smaller than $\delta(U)/4$ if we choose $t<t_U$ for some $t_U$ depending only on $\chi$ and $Y$. The proof of the Lemma is then completed by choosing $\delta_U = \frac{\delta\delta(U)}{2C}.$

\end{proof}
Now, as in \cite{dp}, consider the current $\mathbf{1}_Y\Theta.$ By the Skoda-El Mir extension theorem, $\mathbf{1}_Y\Theta$ is a closed, non-negative $(1,1)$ current supported on $Y$. By standard support theorems,  $\mathbf{1}_Y\Theta = \displaystyle \sum_i \be_i [Y_i]$ for some $\be_i\geq 0$, where $Y_i$ are the irreducible components of $Y$. It can be easily seen using Lemma \ref{pos-mass} that $\be_i>0 \ \forall \ i$. Hence $\Theta \geq \be_Y[Y]$ on all of $M$, where $\be_Y$ can be taken  as $\min_i \be_i$.

Finally, we verify the degenerate cone condition. To this end, we need the following lemma.
\begin{lemma}\label{lem:cone-condition-auxilary}
Let $\chi(x), \Omega(x)$ be smooth K\"ahler forms at a point $x\in M$. Let $\alpha$ be a fixed smooth K\"ahler metric on $M$ such that $\chi(x) \leq C_{\chi} \alpha(x)$ for some $C_{\chi}>0$. Assume that the following degenerate cone condition is met at $x$.
\begin{gather}
n\Omega^{n-1}(x)-\displaystyle \sum_{k=1}^{n-1}c_k k \chi^{n-k}(x) \Omega^{k-1} (x) \geq 0.
\label{ineq:coneconditionassumptioninlemma}
\end{gather}
Let $\beta>0$. There exists an $\epsilon>0$ depending only only on the coefficients $c_k$ and the constants $C_{\chi}, \beta$ such that $\Omega'(x)=\Omega(x)+2\beta\alpha(x)$ satisfies the $\epsilon$-uniform cone condition.
\begin{gather}
n(1-\epsilon)(\Omega')^{n-1}(x)-\displaystyle \sum_{k=1}^{n-1}c_k k \chi^{n-k}(x) (\Omega')^{k-1}(x) \geq 0.
\label{ineq:epsilonuniforminlemma}
\end{gather}
\label{lem:strictcone}
\end{lemma}
\begin{proof}
Choose coordinates near $x$ such that $\chi(x)$ is Euclidean and $\Omega(x)$ is diagonal with eigenvalues $\lambda_i$. Then the cone condition can be written as
\begin{gather}
1\geq \sum_{k=1}^{n-1}\frac{c_k}{{n \choose k}}  S_{n-k;j}\left (\frac{1}{\lambda} \right ) \ \forall \ 1\leq j\leq n.
\label{ineq:coneconditionassumptioninproofoflemma}
\end{gather}
Note that $\Omega'(x)\geq \Omega(x)+2\beta \frac{\chi(x)}{C_{\chi}}$. Let $\tilde{\lambda}_i=\lambda_i+\frac{\beta}{C_{\chi}} \geq\gamma$ where $\gamma=\frac{\beta}{C_{\chi}}$. By monotonicity, $\{\tilde{\lambda}_i\}$ also satisfies Inequality \ref{ineq:coneconditionassumptioninproofoflemma}. Let $\lambda'_i = \tilde{\lambda}_i+\gamma$. If we prove that $\{\lambda'_i\}$ satisfies the $\epsilon$-uniform cone condition for some $\epsilon$, then again by monotonicity, so does $\Omega'$. \\
\indent To this end, we first note that
\begin{gather}
\frac{1}{\lambda'_i} \geq \frac{1}{2\tilde{\lambda}_i} \ \forall \ 1\leq i \leq n.
\label{ineq:comparison}
\end{gather}
Define $$B:=\sum_{k=1}^{n-1}\frac{c_k}{{n \choose k}}  S_{n-k;j}\left (\frac{1}{\tilde{\lambda}} \right ).$$ Using \ref{ineq:comparison}, $\ \forall \ 1\leq j\leq n $ we see that the following holds.
\begin{align}
\sum_{k=1}^{n-1}\frac{c_k}{{n \choose k}}  S_{n-k;j}\left (\frac{1}{\lambda'} \right ) &\leq \sum_{k=1}^{n-1}\frac{c_k}{{n \choose k}}  S_{n-k;j}\left (\frac{1}{\tilde{\lambda}} \right ) +\sum_{k=1}^{n-1}\frac{c_k}{{n \choose k}} \Bigg ( S_{n-k;j}\left (\frac{1}{\lambda'}  \right )  - S_{n-k;j}\left(\frac{1}{\tilde{\lambda}} \right) \Bigg )\nonumber \\
&\leq B-\sum_{k=1}^{n-1}\frac{c_k \gamma^{n-k}}{2^{n-k}{n \choose k}} S_{n-k;j}\left(\frac{1}{\tilde{\lambda}^2} \right) \nonumber \\
&\leq B-\frac{B^2}{C},
\end{align}
where $C>4$ is some positive constant depending only on the coefficients $c_k$, and the constants $n, \gamma$. Since $0<B< 1$, we see that for $\epsilon=\frac{1}{C}$, the above expression is $\leq 1-\epsilon$, as desired.
\end{proof}
 Let $U$ be a coordinate neighbourhood, and $\Theta\Big|_{U} = \ddbar \varphi,~\tilde{\Omega}_{t}\Big|_U = \ddbar\varphi_{t} $.  Let $\varphi_\delta$ and $\varphi_{t,\delta}$ be the convolutions of $\varphi$ and $\varphi_t$ with the standard molifier.  Firstly, recall that the ($\vep$)-cone condition for a smooth positive form $\gamma$ can be written as
\begin{gather}
\displaystyle \sum_{k=1}^{n-1}c_k \frac{1}{{n \choose k}} S_{n-k;i} \left(\frac{1}{\lambda_{\gamma}} \right) \leq 1-\vep,
\label{degenconesmooth}
\end{gather}
where $\lambda_{\gamma}$ are the eigenvalues of $\gamma$ with respect to $\chi$. Let $\chi_0$ be a constant coefficient form in $U$ such that $\chi \geq \chi_0$ on $B_{\delta}(x)$. Choose a strictly positive sequence $\mu_i\rightarrow 0$. By Lemma \ref{lem:strictcone} we see that there exists a corresponding sequence of $\vep_i\rightarrow 0$ such that $\spbp \varphi_{t}+\mu_i \chi$ satisfies the $(\vep_i-)$ uniform cone condition. By monotonicity, $\spbp \varphi_{t}+\mu_i\chi$ satisfy the $(\vep_i-)$ uniform cone condition at $x$ with $\chi$ replaced with $\chi_0$.  Since averaging over a group and convolutions are basically convex linear combinations and the expression on the left-hand side of \ref{degenconesmooth} is convex, using Jensen's inequality we see that $\spbp \varphi_{t,\delta}+\mu_i\chi_{,\delta}$ satisfies the degenerate cone condition at $x$ with $\chi_0$ replacing $\chi$. Moreover, since $\Omega_{t} \rightarrow \Theta$ weakly, and convolution is basically testing against a smooth function, we see that $\spbp \varphi_{t,\delta}+\mu_i\chi_{,\delta}$ converges pointwise to $\spbp \varphi_{\delta}+\mu_i\chi\chi_{,\delta}$ as $t\rightarrow 0$. Hence $\spbp  \varphi_{\delta}+\mu_i\chi_{,\delta}$ satisfies the $(\vep_i-)$ uniform cone condition at $x$. Therefore, $\Theta$ satisfies the degenerate cone condition as in Definition \ref{degenconedef}.
\end{proof}

\section{Proof of Theorem \ref{thm:main}}\label{firstmaintheorem}
\indent In this section, we denote by $E_c(T)$, the Lelong subvariety $\{x \vert \nu(T,x) \geq c>0\}$. Note that if $T$ is $G$-invariant, then so is $E_c(T)$ for all $c$. Our goal is is to prove Theorem \ref{thm:main}. First we need the following gluing proposition. 
\begin{proposition}\label{prop:gluing}
Let $(M,\chi)$ be a compact K\"ahler manifold, $[\Omega_0]$ be another K\"ahler class, $Y$ be an ample divisor. Let $0<\be<1$, and $0<\vep_{\ref{prop:gluing}}<\frac{\be}{1000c_n}$, where $c_n$ is the dimensional constant defined below. Suppose $T \geq \be[Y]$ is a positive  current in $[(1-\vep_{\ref{prop:gluing}})\Omega_0 ]$ satisfying the strict cone condition on $M\setminus Y$. Then there exists a $c_{\ref{prop:gluing}}$ with the following property: If $U$ is a neighbourhood of $Z:=E_{c_{\ref{prop:gluing}}} (T) \cup Y$ with a K\"ahler form $\theta = \Omega_0\Big|_U + \ddbar\psi_U$ satisfying the cone condition on $U$, then there exists a K\"ahler form $\hat \Omega_0 = \Omega_0 + \ddbar \hat\psi$ satisfying the cone condition, i.e., $$n\hat\Omega_0^n - \sum_{k=1}^{n-1} c_k k \chi^{n-k} \hat\Omega_0^{k-1} >0.$$ 
\end{proposition}

\begin{proof}
We follow the line of argument in \cite{gchen} closely, with one necessary addition. The basic idea is to cover the manifolds with balls of a small but definite size on which the current can be written as the $\ddbar$ of some potentials, and then to glue these potentials to $\psi_U$ using the modification of the Richberg technique \cite{Richberg} due to Blocki-Kolodziej \cite{blockikol}, and its improvement (to allow small Lelong numbers) in \cite{gchen}. The additional complication in our case, as opposed to the method in \cite{gchen}, is that $T$ satisfies the cone condition only on $M\setminus Y$, while in \cite{gchen}, the cone condition is satisfied on all of $M$. So we need to make sure that during the gluing process, the metric $\theta$ is unchanged in a neighbourhood of $Y$. 
\begin{itemize}
\item We cover $M$ with a finite number of balls $B^{i}_{r} = B_r(x_i)$ of radius $r<1$ (with respect to $\Omega_0$) centred at $x_i$ such that the following conditions hold:
\begin{enumerate}
\item For every $z\in Y$, there exists an $i$, such that $B_{r/4}(z) \subset B_r^i$, where $B_{r/4}(z)$ is the $\Omega_0$-ball of radius $\frac{r}{4}$ around $z$.
\item  On $B_{2r}^i:= B_{2r}^i(x_i)$, we have $\Omega_0 = \ddbar\varphi^i_0$, and 
\begin{align*}
\chi_0^i\leq \chi \leq \left(1+\frac{\vep_{\ref{prop:gluing}}}{1000n^{10n}} \right)\chi_0^i\\
|\varphi_0^i - |z|^2| < \vep_{\ref{prop:gluing}}r^2,
\end{align*}
where $\chi_0^i$ is a K\"ahler form  on $B_{2r}^i$ with constant coefficients. Moreover, choose the holomorphic coordinates on the ball to be such that  $\Omega_0$ is within $1+\frac{1}{1000}$ of the Euclidean metric. 
\end{enumerate}
It is clear that the second condition can be achieved by choosing a sufficiently refined cover. For the first condition, we simply include a finitely many balls $\{B_r^{i} := B_{r}(z_i)\}_{i=1}^N$ centred at points $z_i\in Y$ such that $Y\subset \cup_{i=1}^NB_{r/4}(z_i)$.

Next we let $\varphi_\delta^i$ be the smoothening (by convolution) of  $\varphi^i:=\varphi_T + (1-\vep_{\ref{prop:gluing}})\varphi^i_0$, where $T = (1-\vep_{\ref{prop:gluing}})\Omega_0 + \ddbar\varphi_T$. By our hypothesis, $T = \ddbar \varphi^i$ satisfies the (strict) cone condition on $M\setminus Y$. By convexity, it then follows that $\ddbar\varphi_\delta^i$ also satisfies the cone condition $$n(\ddbar \varphi_\delta^i)^{n-1}- \sum_{k=1}^{n-1} c_k k \chi^{n-k} (\ddbar\varphi_\delta^i)^{k-1} > 0,$$ at every $x\in B_r^i$ so long as $B_\delta(x)$ does not intersect $Y$. 
\item Recall that the Lelong number at level $\delta<r/4$ is defined by $$\nu^i(x,\delta) = \frac{\hvarphi^i_{\frac{r}{4}}(x) - \hvarphi_{\delta}^i}{\log \frac{r}{4}- \log\delta},$$ where $\hvarphi_\delta^i(x) = \sup_{B_\delta^i}\varphi^i$. One can show that $\nu^i(x,\delta)$ is increasing in $\delta$, and the Lelong number of $T$ at $x$ (which of course is independent of the potential) is given by $$\nu(T,x) = \lim_{\delta\rightarrow 0^+}\nu^i(x,\delta).$$ We now let $c_{\ref{prop:gluing}}:=c_n\vep_{\ref{prop:gluing}}r^2$, where 
\begin{equation}\label{eq:cn}
c_n = \frac{2}{\Big(|\mathbb{S}^{2n-1}|\int_0^1\log t^{-1}\mathrm \rho(t) t^{2n-1}\,dt + \frac{3^{2n-1}}{2^{2n-3}}\Big)}.
\end{equation}  We let $Z := E_{c_{\ref{prop:gluing}}}\cup Y$. By our choice of $\vep_{\ref{prop:gluing}}$, we have that $c_{\ref{prop:gluing}} < \frac{\be}{1000}$, and so in particular, $Y\subset E_{c_{\ref{prop:gluing}}}$. Also note that our constant is $10\epsilon_{4.5}$, where $\epsilon_{4.5}$ is the constant in Proposition 4.1 in \cite{gchen}. Now let $\overline{U'}\subset U$.  For $\delta$ small (in particular smaller than $\delta_0$ in Lemma \ref{lem:largelelong} below), if we  define $\hat\psi$ to be the regularised maximum of $\psi_U + \frac{c_{\ref{prop:gluing}}}{4}\log\delta$ on $\overline{U'}$ and $\varphi_\delta^i - \varphi_0^i$, then by \cite[Proposition 4.1]{gchen}, $\hat\psi$ will be a smooth $\Omega_0$-PSH function. We let $\hat\Omega = \Omega_0 + \ddbar\hat\psi$. To complete the proof of the proposition we need to make sure that the cone condition is satisfied. Recall that $T$ satisfies the cone condition on $M\setminus Y$, and hence the smoothening $\ddbar\varphi_{\delta}^i$ satisfies the cone condition at $x$ if $B_\delta(x)$ does not intersect $Y$. On the other hand $\theta$ satisfies the cone condition everywhere in $U$. If $B_\delta(x)$ does intersect $Y$ for some $\delta<\delta_0$, where $\delta_0$ is as in Lemma \ref{lem:largelelong} below, then in particular $d_{\Omega_0}(x,Y)<4\delta$, and so by Lemma \ref{lem:largelelong}, $\nu^i(x,\delta) > c_{\ref{prop:gluing}}/2$ for some $i$. Once again by \cite[Proposition 4.1]{gchen}, $$\max_{B_r^i}(\varphi_\delta^i - \varphi_0^i) \leq \inf_{\overline{U'}}\Big(\psi_U + \frac{c_{\ref{prop:gluing}}}{4}\log\delta - \vep_{\ref{prop:gluing}}r^2\Big).$$ In particular, $\hat\psi(x) = \psi_U(x) + \frac{c_{\ref{prop:gluing}}}{4}\log\delta$, and hence $\hat\Omega_0$ satisfies the cone condition at $x$. 
\end{itemize}
\end{proof}

\begin{lemma}\label{lem:largelelong}
There exists a $\delta_0<r/4$, such for all $\delta<\delta_0$, there exists $U_{5\delta} \subset \bar{U}_{5\delta} \subset U'$, an $\Omega_0$-tubular neighbourhood of $Y$ of size $5\delta$,  with the property that for all  $x\in U_{5\delta}$, there exists an $i$ such that $x\in B^i_r$ and $\nu^i(x,\delta)>c_{\ref{prop:gluing}}/2$.
\end{lemma}
\begin{proof} 
\indent For any plurisubharmonic function $\varphi$, we let $\nu_\varphi(x):= \nu(\ddbar\varphi,x)$, and  let $\nu_\varphi(x,\delta)$ be the corresponding Lelong number at level $\delta$. In each of the balls $B_{2r}(x_i)$ (where $x_i \in Y$), assume that $Y$ is given by $f_i=0$ where $f_i$ is an analytic function. Let
$\phi^i _1 = \frac{\beta}{10} \ln\vert f_i \vert^2$ and $\phi^i_2=\varphi^i-\phi_1^i$. \\
\indent Note that $\nu_{\varphi^i}(x) \geq  \nu_{\phi^i_1}(x)$, since $T\geq \be[Y]$. Assume $\delta<r/4<1$. Now, since $\widehat{(\psi_1+\psi_2)_{\delta}} \leq \widehat{(\psi_1)_{\delta}}+\widehat{(\psi_2)_{\delta}}$ for any two upper-semicontinuous functions $\psi_1, \psi_2$, we see that the following holds.
\begin{align}
\nu_{\varphi^i}(x,\delta) = \frac{\widehat{\varphi^i_{r/4}}(x)-\widehat{\varphi^i_{\delta}}(x)}{\ln(r/4)-\ln(\delta)} &\geq \frac{\widehat{\varphi^i_{r/4}}(x)}{\ln(r/4)-\ln(\delta)}-\frac{\widehat{(\phi^i_1)_{r/4}}(x)}{\ln(r/4)-\ln(\delta)}-\frac{\widehat{(\phi^i_2)_{r/4}}(x)}{\ln(r/4)-\ln(\delta)}
\nonumber \\
&+ \nu _{\phi^i_1}(x,\delta)+\nu _{\phi^i_2}(x,\delta).
\end{align}
 Note that $\nu_{\phi^i_2}(x,\delta) \geq \nu_{\phi^i_2}(x)>0$ because the $\delta$-Lelong number increases with $\delta$. Moreover as $\delta\rightarrow  0$, the first three terms go to zero uniformly (in $x, i$), and so  if $\delta_0$ is sufficiently small, then for every $i$,
\begin{gather}
\nu _{\varphi^i}(x,\delta)\geq -\frac{\beta}{100000}+\nu _{\phi^i_1}(x,\delta).
\end{gather} 
For an even smaller $\delta_0$, 
\begin{gather}
\nu _{\varphi^i}(x,\delta)\geq -\frac{\beta}{5000}+\frac{\beta}{11} \frac{\sup_{B_{x,\delta}}(\ln(\vert f_i \vert^2))}{\ln(\delta)} 
\label{almosttherefornu} 
\end{gather}
Now we use the assumption that $x\in U_{5\delta_0}$. Hence, there exists a point $z_x \in Y$ such that $d_{\Omega_0}(x,z_x)\leq 5\delta_0$. This point is at a Euclidean distance (in some coordinate chart $B_{2r}(x_j)$) of at most $6\delta_0$ from $z_x$. Note that $f_j(z_x)=0$. Hence, $$\vert f_j(x)-f_j(z_x) \vert^2 \leq C \vert x-z_x \vert^2 \leq C 36 \delta^2,$$ where $C$ depends on $\max_{\bar{B}_r(x_j)} \vert \nabla_{\Omega_0} f_j \vert^2$. Hence, for a sufficiently small $\delta_0$, Inequality \ref{almosttherefornu} implies the following estimate that completes the proof.
\begin{gather}
\nu_{\varphi^j}(x,\delta)\geq \frac{\beta}{20} >\frac{c_{\ref{prop:gluing}}}{2}.
\end{gather}
\end{proof}

In what follows it is convenient to adopt the following notation: For $m<n$, and $j=0,1,\cdots,m-1$, we set
\begin{equation}\label{eq:b}
b_j := \frac{c_{j+n-m}{j+n-m\choose n-m}}{{n\choose m}}.
\end{equation}
\begin{lemma}\label{lem:local}
Let $Z$ be a smooth $m$-dimensional sub-variety of $M$. Suppose there exists a K\"ahler metric $\omega_Z = \Omega_0\Big|_{Z} + \ddbar\psi_Z$ on $Z$ such that $$\Big(1-\frac{\vep}{2}\Big)\omega_Z^{m} - \sum_{j= 0}^{m-1}b_j\chi^{m-j}\omega_Z^{j}>0, $$ for some $\vep>0$. Then there exists a neighbourhood $U$ of $Z$, and a K\"ahler form $\Omega_U = \Omega_0 + \ddbar\psi_U$ on $U$ such that $$n\Big(1 - \frac{\vep}{2}\Big)\Omega_{U}^{n-1} - \sum_{k=1}^{n-1} c_k k \chi^{n-k} \Omega_{U}^{k-1} >0.$$ 
\end{lemma}
Note that the hypothesis, in particular implies that $\omega_Z$ satisfies the cone condition on $Z$. But to construct a K\"ahler metric on a neighbourhood $U$ of $Z$, we need this slightly stronger condition. 

\begin{proof} 
Firstly, by compactness of $Z$ we can assume that 
\begin{gather}
\Big(1-\frac{\vep}{2}-\delta_Z\Big)\omega_Z^{m} - \sum_{j= 0}^{m-1}b_j\chi^{m-j}\omega_Z^{j}>0,
\label{ineq:hypothesislemma}
\end{gather}
 for some $\delta_Z>0$.\\
\indent Let $(U_{\al},\{z_{\al}^{i}\})$ be finitely many coordinate charts in $M$ that cover $Z$ and such that $Z\cap U_{\al}$ is given by $z_{\al}^{m+1}=z_{\al}^{m+2}=\ldots=0$. Let $\rho_{\al} (z_{\al}^1, z_{\al}^2,\ldots, z_{\al}^m) $ be a partition-of-unity on $Z$ subordinate to $U_{\al}\cap Z$. Let $\pi^{\al}$ be the projection to the first $m$ coordinates. Let $\eta_{\al}(z_{\al}^{m+1},\ldots)$ be a smooth function with compact support in the ``vertical part" of $U_{\al}$ such that it is identically $1$ in a neighbourhood of the origin. Then, $\tilde{\rho}_{\al}(z_{\al})=\eta_{\al}(\pi^{\al})^{*}\rho_{\al}$ is a partition-of-unity on a closed subset of $\cup_{\al} U_{\al}$. (Here, we extend $\tilde{\rho}_{\al}$ globally by declaring it as $0$ outside $U_{\alpha}$.) Abusing notation slightly, we let $$\Omega_U = \tilde{\Omega}_U + C\ddbar d_{\Omega_0}(\cdot,Z)^2,$$ where $C\geq 1$ is some large constant and $$\tilde{\Omega}_U=\Omega_0 + \ddbar\left(\displaystyle \sum_{\al} \tilde{\rho}_{\al} (\pi^{\al})^{*}\psi_Z\right).$$ We note that while $d_{\Omega_0}(\cdot, Z)^2$ is only Lipschitz globally, it is smooth near $Z$. For small $U$, the form $\Omega_U$ is a K\"ahler metric. By continuity, it is enough to verify the cone condition on $Z$. \\
\indent Near an arbitrary point $p\in Z$ choose a coordinate chart $(U_{\alpha_0},z^i)$ containing $p$. Extend the tangential coordinate vector fields $\frac{\partial}{\partial z^1},\ldots,\frac{\partial}{\partial z^m}$ using an $\Omega_0$-orthonormal frame $e_{m+1},\ldots,e_n$ that are orthogonal to the tangential vector fields. For the remainder of this proof, all constants are uniform in $p$ (but depend on the finitely many fixed coordinate charts) and all forms are evaluated at $p$. \\
\indent Note that $$\Omega_U \left(\xi, \bar{\xi} \right )  \geq \tilde{\Omega}_U \left(\xi,\bar{\xi} \right )=\omega_Z\left(\xi,\bar{\xi} \right )$$ for all tangential directions $\xi$. Moreover, $\spbp d_{\Omega_0}(\cdot,Z)^2\geq \frac{1}{K}\displaystyle \sum_{i=m+1}^n e^i\wedge \bar{e}^i$ for some $K>1$. Thus, for any given $0<\delta'<1$, we may choose $C=C_{\delta'}>1$ large enough so that $$\Omega_U \geq \Omega_{\delta',U} =(1-\delta') \left( (\pi^{\al})^* \omega_Z + \sqrt{C_{\delta'}}\displaystyle \sum_{i=m+1}^n e^i\wedge \bar{e}^i \right),$$ and $$\chi \leq \chi_{\delta',U}=(1+\delta') \left((\pi^{\al})^*\chi + (C_{\delta'})^{1/3}\displaystyle \sum_{i=m+1}^n e^i\wedge \bar{e}^i\right).$$
\indent It is enough to prove that $\Omega_{\delta',U}$ satisfies the $\frac{\epsilon}{2}$-uniform cone condition with respect to $\chi_{\delta',U}$. Now suppose the eigenvalues of $\Omega_U$ at $p$ (with respect to $\chi$) are given by $\la_1,\cdots,\la_n$ (where $\lambda_1,\ldots, \lambda_m$ correspond to the invariant subspace spanned by the tangential directions). The inequality \ref{ineq:hypothesislemma} implies that
\begin{equation}
1-\frac{\vep}{2}-\delta_Z>\sum_{j=0}^{m-1} b_j \left(\frac{1-\delta'}{1+\delta'}\right)^{m-j}  \frac{j! (m-j)!}{m!} S_{m-j} \left(\frac{1}{\lambda_1},\cdots,\frac{1}{\la_m}\right).
\end{equation}  
\indent Choose $\delta'$ small enough (by choosing a large $C_{\delta'}$) so that
\begin{equation}\label{hypothesisb}
1-\frac{\vep}{2}-\frac{\delta_Z}{2}>\sum_{j=0}^{m-1} b_j  \frac{j! (m-j)!}{m!} S_{m-j} \left(\frac{1}{\lambda_1},\cdots,\frac{1}{\la_m}\right).
\end{equation}  
By choosing an even larger $C_{\delta'}$, we may assume that for $k=m+1,\cdots,n$, $\frac{1}{\la_k}<\frac{1}{\Lambda}$ for some large $\Lambda$. So we only need to control the terms in $S_{n-k;i}(1/\la)$ that  have no  $\la_k$  for $k\geq m+1$. In particular, we only need to worry about the case when $1\leq n-k \leq m$ or equivalently $n-m\leq k\leq n-1$. That is, it is enough to prove that $$\sum_{k=n-m}^{n-1}\frac{c_k}{{n\choose k}}S_{n-k}\Big(\frac{1}{\la_1},\cdots,\frac{1}{\la_m}\Big)<1-\frac{\vep}{2}-\frac{\delta_Z}{2}.$$ But this is precisely \eqref{hypothesisb}.
\end{proof}

\subsection*{Proof of Theorem \ref{thm:main}}
The implication $(1)\iff (2)$ is proved in Lemma \ref{necessity} and $(2)\implies (3)$ is trivial, and hence we only focus on $(3)\implies(2)$. We proceed by induction on the dimension of $M$. The theorem is certainly true for $n=1$. Suppose the theorem is true for dimensions less than $n$. The key in making the inductive step work is the following observation. 

\begin{lemma}\label{lem:localconesmooth}
Let $Z\subset M$ be a smooth sub-variety  of dimension $m<n$. Assuming $(3)$ in the statement Theorem \ref{thm:main}, there exists a K\"ahler metric $\omega_Z = \Omega_0\Big|_Z + \ddbar\psi_Z$ on $Z$ such that $$\Big(1-\frac{\vep}{2}\Big)\omega_Z^{m} - \sum_{j= 0}^{m-1}b_j\chi^{m-j}\omega_Z^{j}>0.  $$ As a consequence, by Lemma \ref{lem:local}, there exists a neighbourhood $U$ of $Z$ and a K\"ahler metric $\Omega_U = \Omega_0 + \ddbar\psi_U$ on $U$ such that $$n\Big(1 - \frac{\vep}{2}\Big)\Omega_{U}^{n-1} - \sum_{k=1}^{n-1} c_k k \chi^{n-k} \Omega_{U}^{k-1} >0.$$ 

\end{lemma}

\begin{proof} When there is no scope for confusion, we continue to denote  $\Omega_0\Big|_Z$ and $\chi\Big|_Z$ by $\Omega_0$ and $\chi$ respectively. Let $a_Z$ such that $$a_Z\int_{Z}\chi^m = \int_Z {n\choose  m}\Omega_0^{m}  -  \sum_{j=0}^{m-1}c_{j+n-m}{j+n-m\choose n-m}\int_{Z}\chi^{m-j}\Omega_0^{j}.$$ Then by condition $(3)$, $a_Z>0$.  In fact, $$a_Z> \vep_Z:= \vep{n\choose m}\frac{\int_Z\Omega_0^m}{\int_Z\chi^m}. $$ For $\omega\in [\Omega_0\Big|_Z]$, with  $b_j$ given by\eqref{eq:b}, consider the equation 
\begin{equation}\label{eq:solveb}
\Big(1-\frac{\vep}{2}\Big)\omega^m =  \sum_{j=1}^{m-1}b_j\chi^{m-j}\omega^{j} + f\chi^{m},
\end{equation}where $$ f = \frac{a_Z + c_{n-m}-\vep_Z/2}{{n\choose m}}>0.$$ Then clearly both sides of the  equation integrate out to the same quantity. 

\noindent{\bf Claim.} For any sub-variety $V\subset Z$ of co-dimension $p$, we have  $$\int_V\Big({m\choose p}\Omega_0^{m-p} -  \sum_{j=p}^{m-1}b_j{j\choose p}\chi^{m-j}\Omega_0^{j-p}>\vep\int_V{m\choose p}\Omega_0^{m-p},$$ where $\vep$ is the same as in condition $(3)$.

\noindent{\em Proof of the Claim.} Since $V$ is of co-dimension $p+n-m$ in $M$, by condition $(3)$ we have $$\int_V\Big(\Omega_0^{m-p} - \sum_{k=p+n-m}^{n-1}c_k\frac{{k \choose p+n-m}}{{n\choose p+n-m}}\chi^{n-k}\Omega_0^{k- (p+n-m)}\Big) > \vep\int_V\Omega_0^{m-p}.$$ But then we compute that 
\begin{align*}
\int_V\Omega_0^{m-p} -  \sum_{j=p}^{m-1}b_j\frac{{j\choose p}}{{m\choose p}}\chi^{m-j}\Omega_0^{j-p} &= \int_V\Omega_0^{m-p} -  \sum_{j=p}^{m-1}c_{j+n-m}\frac{{j\choose p}{j+n-m\choose n-m}}{{m\choose p}{n\choose m}}\chi^{m-j}\Omega_0^{j-p} \\
&=  \int_V\Omega_0^{m-p} -  \sum_{k=p+n-m}^{n-1}c_{k}\frac{{k-(n-m)\choose p}{k\choose n-m}}{{m\choose p}{n\choose m}}\chi^{m-j}\Omega_0^{j-p}\\
&= \int_V\Omega_0^{m-p} -  \sum_{k=p+n-m}^{n-1}c_{k}\frac{{k\choose p+n-m}}{{n\choose p+n-m}}\chi^{m-j}\Omega_0^{j-p}\\
&\geq \vep\int_V\Omega_0^{m-p}.
\end{align*}
From the claim it follows that for any sub-variety $V\subset Z$ of co-dimension $p$ in $Z$, we have $$\int_V\Big({m\choose p}\Omega_0^{m-p} -  \Big(1-\frac{\vep}{2}\Big)^{-1}\sum_{j=p}^{m-1}b_j{j\choose p}\chi^{m-j}\Omega_0^{j-p}>\frac{\vep}{2-\vep}\int_V{m\choose p}\Omega_0^{m-p}.$$ Then by the induction hypothesis and Theorem \ref{purePDEthm} there exists a metric $\omega_Z\in [\Omega_0\Big|_Z]$ solving \eqref{eq:solveb}. In particular, since  $a_Z - \vep_Z/2>0$, we have that  $$\Big(1-\frac{\vep}{2}\Big)\omega^{m}_Z -  \sum_{j=0}^{m-1}b_j\chi^{m-j}\omega_Z^{j}> 0.$$ 
\end{proof}

We now return to the proof of Theorem \ref{thm:main}. We use a continuity method.  For $t\geq 0$, consider the following family of equations depending on $t$.
\begin{gather}
\Omega_{t}^n = \sum_{k=1}^{n-1} c_k \chi^{n-k} \Omega_{t}^{k}+ a_t\chi^n +f\chi^n,
\label{cont-genmaeqn}
\end{gather}
where $\Omega_t\in (1+t)[\Omega_0]$. Note that $a_t\xrightarrow{t\rightarrow 0^+}0$, and hence we want to  solve the above equation at $t=0$. For $t>>1$, $\hat \Omega_t =  (1+t)\Omega_0$  satisfies the cone condition   $$n\hat \Omega_{t}^{n-1} - \sum_{k=1}^{n-1} c_k k \chi^{n-k}\hat \Omega_{t}^{k-1} >0.$$ By Theorem \ref{purePDEthm}, there exists a solution to \eqref{cont-genmaeqn} for $t>>1$. In particular, if we let $I = \{t\in [0,\infty)~|~ \eqref{cont-genmaeqn} \text{ has a solution}\}$, then $I$ is non-empty. By Lemma \ref{necessity}, the cone condition is preserved and hence the linearisation of Equation \ref{cont-genmaeqn} is elliptic. The infinite-dimensional implicit function theorem easily implies that $I$ is open. We need to show that the set is closed. Moreover, by the nature of the cone condition, if $t\in I$, then $t'\in I$ for all $t'>t$. Let $t_0 = \inf I.$ It is enough to prove that $t_0\in I$. Replacing $\Omega_0$ by $(1+t_0)\Omega_0$, without loss of generality, one can assume that $t_0 = 0$, i.e., we have a solution to \eqref{cont-genmaeqn} for all $t>0$. The strategy is to produce a $G$-invariant K\"ahler form $\hat\Omega_0\in[\Omega_0]$ which satisfies  the cone condition. It is actually enough to produce \emph{some} form satisfying the cone condition. Indeed, by means of averaging over $G$ (with respect to the Haar measure), we get an invariant form. Since the cone condition is convex, it continues to satisfy the cone condition.

\indent Let $Y \subset M$ be a $G$-invariant divisor that contains a K\"ahler metric $\chi_Y$ in its cohomology class. Such a divisor exists by hypothesis of $G$-compatibility. By Theorem \ref{thm:conc-mass}, there exists a $G$-invariant current $\Theta \in [\Omega_0]$ such that $\Theta \geq \be_Y[Y]$ for some $\beta_Y>0$, and $\Theta$ satisfies the degenerate cone condition in the sense of Definition \ref{degenconedef}. Let $\Xi = \Theta - \frac{\be_Y}{2}[Y] + \frac{\be_Y}{2}\chi_Y \in [\Omega_0]$. Then  exists an $\vep_{\ref{prop:gluing}}>0$ such that $T := (1-\vep_{\ref{prop:gluing}})\Xi$ also satisfies the degenerate cone condition on $M\setminus Y$. Indeed, this statement is easily deduced (going back to $\Omega_t$ whose limit is $\Theta$) from Lemma \ref{lem:strictcone}.\\
\indent With the same notation as Proposition \ref{prop:gluing} (with $\be = (1-\vep_{\ref{prop:gluing}})\be_Y/2)$, let $c<c_{\ref{prop:gluing}}$, $E_c (T) = \{x\in M~|~ \nu(T,x)\geq c\}$  and $Z = E_c(T) \cup Y$. Clearly, $Z$ is $G$-invariant. By Proposition \ref{prop:gluing}, the following is enough.
\begin{lemma}\label{localmetric}
There exists a neighbourhood $U$ of $Z$ and a  K\"ahler metric $\Omega_U = \Omega_0 + \ddbar\psi_U $ such that $$n\Omega_U^{n-1} - \sum_{k=1}^{n-1} c_k k \chi^{n-k} \Omega_U^{k-1} >0.$$
\end{lemma}
 \begin{proof}
\indent If $Z$ is smooth, then we are done by applying Lemma \ref{lem:localconesmooth} to each connected component.   In general, $Z$ need not  be smooth. Using a canonical resolution of singularities (as in \cite{bieberstone, Kollar}) there exists a $\pi:\tilde M_r \xrightarrow{\pi_r} \tilde{M}_{r-1}\ldots \xrightarrow{\pi_1} M_0=M$ obtained by successive blowups along smooth $G$-invariant (with respect to the lifted action) centres such that the proper transform $\tilde Z_r$ of $Z$ is smooth. Note that the induction hypothesis is met for every $\tilde M_i$. Firstly, we argue that it is enough to assume that $r = 1$. In the general case, Lemma \ref{num-resolution} shows that the numerical condition is satisfied on each $\tilde M_k$ (and hence in particular, also on $\tilde Z_r$ and each of the centres $\pi_k(\tilde E_k)$, where $\tilde E_k$ are the exceptional divisors). By the discussion below, one starts by constructing a K\"ahler metric on a neighbourhood $\tilde U_r$ of $\tilde Z_r$ in $\tilde M_r$ satisfying the cone  condition. We then proceed recursively. Having constructed a K\"ahler metric in the nieghbourhood $\tilde U_k$ of $\tilde Z_k$, one takes the push forward of the restriction of the metric on $\pi_k(\tilde U_k\setminus \tilde E_k)$, and then again by the discussion below, one glues this metric to a metric in the neighbourhood of $\pi_k(\tilde E_k)$ to obtain a K\"ahler metric satisfying the cone condition in the neighbourhood of $\tilde Z_{k-1}$. Eventually, we obtain a K\"ahler metric  satisfying the cone condition in a neighbourhood of $Z$. 

In view of the remarks above, we now assume that $r = 1$, that is, $E := \pi(\tilde E)$ consists of only  one smooth subvariety. It is a standard fact that $[\pi^*\Omega_0] - t[\tilde E]$ is a $G$-invariant class on $\tilde M$ for all $0<t<<1$. In fact, if $h$ is a hermitian metric on $[\tilde E]$, then for some $C>>1$,  $\tilde\Omega_0:= \pi^*\Omega_0 + C^{-1}\ddbar\log h$ is a K\"ahler metric on $\tilde M$. For $A>>1$, we let $$\begin{cases}\tilde \Omega_s = (1+As)\pi^*\Omega_0 +As^2\tilde\Omega_0\\ \tilde\chi_s :=\pi^*\chi + s^2\tilde\Omega_0.\end{cases}$$ 

To apply our induction hypothesis, we need the following Lemma. 
\begin{lemma}\label{num-resolution} There  exists an $A>>1$, such that for all  $s<<1$, and all $G$-invariant subvarieties $\tilde V\subset \tilde M$ of co-dimension $p$, we have 
 $$\int_{\tilde V}\Big({n\choose p}\tilde\Omega_s^{n-p} - \sum_{k=p}^{n-1}c_k{k\choose p}\tilde \chi_s^{n-k}\tilde \Omega_s^{k-p}\Big) \geq \frac{\vep}{4}{n\choose p}\int_{\tilde V}\tilde\Omega_s^{n-p}.$$
 \end{lemma}
We defer the proof of the Lemma, and first use it to complete the proof of Lemma \ref{localmetric}. By the above Lemma and Lemma \ref{lem:localconesmooth}, for {\em all} $s<<1$, there exists a neighbourhood $\tilde W_1$ of $\tilde Z$ and a K\"ahler metric $\tilde\Omega_{\tilde W_1} = \tilde\Omega_{s} + \ddbar\tilde\psi_{\tilde W_1}$ satisfying $$n\Big(1 - \frac{\vep}{7}\Big)\tilde\Omega_{\tilde W_1}^{n-1} - \sum_{k=1}^{n-1} c_k k \tilde\chi_{s}^{n-k} \tilde\Omega_{\tilde W_1}^{k-1} >0.$$ Similarly, there exists a neighbourhood $\tilde W_2$ of $\tilde E$ and a K\"ahler metric $\tilde\Omega_{\tilde W_2} = \tilde\Omega_{s} + \ddbar\tilde\psi_{\tilde W_2}$ satisfying $$n\Big(1 - \frac{\vep}{7}\Big)\tilde\Omega_{\tilde W_2}^{n-1} - \sum_{k=1}^{n-1} c_k k \tilde\chi_{s}^{n-k} \tilde\Omega_{\tilde W_2}^{k-1} >0.$$ As in \cite{gchen}, we can glue (appropriate modifications of) these metrics to get a K\"ahler metric $\tilde\Omega_{\tilde U} = \tilde\Omega_{s} + \ddbar\tilde\psi_{\tilde U}$ on a neighbourhood $\tilde U$ of $\tilde Z \cup E$ satisfying $$n\Big(1 - \frac{\vep}{8}\Big)\tilde\Omega_{\tilde U}^{n-1} - \sum_{k=1}^{n-1} c_k k \tilde\chi_{s}^{n-k} \tilde\Omega_{\tilde U}^{k-1} >0.$$ So on $W_1 :=\pi(\tilde U\setminus\tilde E)$ we have that $$n\Big(1 - \frac{\vep}{8}\Big)(\pi^{-1})^*\tilde\Omega_{\tilde U}^{n-1} - \sum_{k=1}^{n-1} c_k k \chi^{n-k} (\pi^{-1})^*\tilde\Omega_{\tilde U}^{k-1} >0.$$ Note that $\psi_{W_1}:= \tilde\psi_{\tilde U}\circ\pi^{-1}$ is a bounded function, and moreover, $$(\pi^{-1})^*\tilde\Omega_{\tilde U} = \Omega_0 + As(\Omega_0+ s(\pi^{-1})^*\tilde\Omega_0)+ \ddbar\psi_{W_1}.$$  Now let $$\Omega_{W_1} := \frac{(\pi^{-1})^*\tilde\Omega_{\tilde U}}{1 + As+ As^2} = \Omega_0+ \ddbar\Big(\frac{AC^{-1}s^2\pi_*\log |\tau|_h^2 + \psi_{W_1}}{1+As+ As^2}\Big),$$ where $\tau$ is the defining section of  $\tilde E$. We claim that if $s$ is chosen sufficiently small, then $\Omega_{W_1}$ also satisfies the cone condition. This has to be done a bit carefully since $\Omega_{W_1}$ also depends on $s$. But we notice that 
\begin{align*}
(1+As+ As^2)^{n-1}\Big(n\Big(1- \frac{\vep}{16}\Big)\Omega_{W_1}^{n-1}  &- \sum_{k=1}^{n-1}c_kk\chi^{n-k}\Omega_{W_1}^{k-1}\Big) = n\Big(1 - \frac{\vep}{8}\Big)(\pi^{-1})^*\tilde\Omega_{\tilde U}^{n-1}  \\
 &- \sum_{k=1}^{n-1} c_k k \chi^{n-k} (\pi^{-1})^*\tilde\Omega_{\tilde U}^{k-1}  + \frac{\vep}{16}(\pi^{-1})^*\tilde\Omega_{\tilde U}^{n-1}  \\
 &- \sum_{k=1}^{n-1}c_k k [(1+As+ As^2)^{n-k} - 1]\chi^{n-k}(\pi^{-1})^*\tilde\Omega_{\tilde U}^{k-1}.
\end{align*} 
Now if we choose (and fix) an $s<<1$  such that for all $k=1,\cdots,n-1$, $(1+As+ As^2)^{n-k} - 1 \leq \frac{\vep}{16},$ then we have that $$n\Big(1- \frac{\vep}{16}\Big)\Omega_{W_1}^{n-1}  - \sum_{k=1}^{n-1}c_kk\chi^{n-k}\Omega_{W_1}^{k-1} > 0.$$

Since $\pi(\tilde E)$ is smooth, again by induction hypothesis, we obtain a neighbourhood $W_2$ of $\pi(\tilde E)$ and a K\"ahler metric $\Omega_{W_2} =\Omega_0 + \ddbar \psi_{W_2} $ satisfying the cone condition $$n\Big(1-\frac{\vep}{16}\Big)\Omega_{W_2}^{n-1}- \sum_{k=1}^{n-1} c_k k \chi^{n-k} (\pi^{-1})^*\Omega_{W_2}^{k-1} >0.$$
For a large constant $B$, we let $$\psi_U := \tilde\max\Big(\frac{AC^{-1}s^2\pi_*\log |\tau|_h^2 + \psi_{W_1}}{1+As+ As^2}+B,\psi_{W_2}\Big)$$ on $U = W_1\cup W_2$. Here $\tilde\max$ is the regularized maximum. The resulting metric $\Omega_U:= \Omega_0 + \ddbar\psi_U$ satisfies the required cone condition completing the proof of Lemma \ref{localmetric}, and hence of Theorem \ref{thm:main}, subject to proving Lemma \ref{num-resolution}.
\end{proof}

 We end the section by proving Lemma \ref{num-resolution}.
 
\begin{proof}[Proof of Lemma \ref{num-resolution}] Note that $V  = \pi(\tilde V)$ will also be $G$-equivariant, and so by our hypothesis,  
\begin{align*}
\int_{\tilde V}\Big({n\choose p}(1-{\vep})(1+As)^{n-p}\pi^*\Omega_0^{n-p} - \sum_{k=p}^{n-1}c_k{k\choose p}((1+As)\pi^*\chi)^{n-k}((1+As)\pi^*\Omega_0)^{k-p}\Big)\\
 = (1+As)^{n-p}\int_V\Big({n\choose p}(1-{\vep})\Omega_0^{n-p} - \sum_{k=p}^{n-1}c_k{k\choose p}\chi^{n-k}\Omega_0^{k-p}\Big) \geq 0. 
\end{align*}
Hence it suffices to prove that
 \begin{align*}
 \int_{\tilde V}\Big({n\choose p}(1-\frac{\vep}{4})\tilde\Omega_s^{n-p}& - \sum_{k=p}^{n-1}c_k{k\choose p}\tilde \chi_s^{n-k}\tilde \Omega_s^{k-p}\Big) \\
& \geq \int_{\tilde V}\Big({n\choose p}(1-\frac{\vep}{2})(1+As)^{n-p}\pi^*\Omega_0^{n-p}  \\
&-\sum_{k=p}^{n-1}c_k{k\choose p}((1+As)\pi^*\chi)^{n-k}((1+As)\pi^*\Omega_0)^{k-p}\Big).
 \end{align*}
 Since the integrals are only  dependent on the cohomology classes, exactly as in \cite{gchen}, we replace $\Omega_0$ by a  more suitable metric to carry out the above estimate. By assumption $E = \pi(\tilde E)$ is a smooth complex sub-variety  of dimension $q$, and hence by Lemma  \ref{lem:local} there exists a neighbourhood $U$ of $E$ and a K\"ahler form $\Omega_U = \Omega_0 + \ddbar\psi_U$ such that $$n\Big(1 -  \frac{3\vep}{4}\Big)\Omega_{U}^{n-1} - \sum_{k=1}^{n-1} c_k k \chi^{n-k} \Omega_{U}^{k-1} >0.$$ 
 Note that for any $q$, 
 \begin{equation}\label{eqn:localconeest}
 {n\choose q}\Big(1-\frac{3\vep}{4}\Big)\Omega_U^{n-q} - \sum_{k=q}^{n-1}c_k{k\choose q}\chi^{n-k}\Omega_U^{k-q} > 0.
 \end{equation}
 On $\tilde M\setminus  \tilde E$, $\tilde\Omega_0 = \pi^*\Omega_0 + C^{-1}\ddbar\log|\tau|_h^2 $, where $\tau$ is the defining section of $\tilde E$. We then let $\psi_1$ be the regularised maximum of  $\psi_U$ and $C^{-1}\log|\tau|_h^2 + C'$ for some $C'>>1$. Then $\Omega_1 = \Omega_0 + \ddbar\psi_1$ is K\"ahler on $M$, and $\Omega_1\Big|_{U'}  = \Omega_{U}\Big|_{U'}$ for some $U'\subset\subset U$. We then let $$\tilde\Omega_{1,s} := (1+ As)\pi^*\Omega_1 +As^2\tilde\Omega_0$$ It is enough to prove the following: 
 \begin{align}\label{claim}
{n\choose p}(1-\frac{\vep}{4})\Big(\tilde\Omega_{1,s}^{n-p} &-(1+As)^{n-p}\pi^*\Omega_1^{n-p}\Big)\\
&\geq  \sum_{k=p}^{n-1}c_k{k\choose p}\Big(\tilde \chi_s^{n-k}\tilde \Omega_{1,s}^{k-p}-((1+As)\pi^*\chi)^{n-k}((1+As)\pi^*\Omega_1)^{k-p}\Big)\nonumber
 \end{align}
We first focus on points $x\in \pi^{-1}(U')$. Note that in this region $\Omega_1$ satisfies \eqref{eqn:localconeest}. Also for calculations in the region $\pi^{-1}(U')$, it is more convenient to work with $\tilde \chi_{1,s}: = (1+As)\pi^*\chi + s^2\tilde\Omega_0.$ Since $\tilde\chi_{1,s}\geq \tilde\chi_{s}$, it suffices to prove that for all $x\in \pi^{-1}(U')$,  \begin{align}\label{claim2}
{n\choose p}(1-\frac{\vep}{4})\Big(\tilde\Omega_{1,s}^{n-p} &-(1+As)^{n-p}\pi^*\Omega_1^{n-p}\Big)\\
&\geq  \sum_{k=p}^{n-1}c_k{k\choose p}\Big(\tilde \chi_{1,s}^{n-k}\tilde \Omega_{1,s}^{k-p}-((1+As)\pi^*\chi)^{n-k}((1+As)\pi^*\Omega_1)^{k-p}\Big).\nonumber
 \end{align} We call the term on the left as $L$ and the term on the right as $R$, and expand: 
 \begin{align*}
 L-R &= {n\choose p}(1-\frac{\vep}{4})\sum_{q = p+1}^{n}{n-p\choose n-q}(1+As)^{n-q}\pi^*\Omega_1^{n-q}(As^2\tilde\Omega_0)^{q-p}\\
 &-  \sum_{k=p}^{n-1}\sum_{a=0}^{n-k}\sum_{b=0}^{k-p}c_k{k\choose p}{n-k\choose a}{k-p\choose b}A^{k-p-b}(1+As)^{a+b}(\pi^*\chi)^a(\pi^*\Omega_1)^b(s^2\tilde\Omega_0)^{n-p-a-b} \\
& +\sum_{k=p}^{n-1}c_k{k\choose p}((1+As)\pi^*\chi)^{n-k}((1+As)\pi^*\Omega_1)^{k-p}\\
&= {n\choose p}(1-\frac{\vep}{4})\sum_{q = p+1}^{n}{n-p\choose n-q}(1+As)^{n-q}\pi^*\Omega_1^{n-q}(As^2\tilde\Omega_0)^{q-p}-S_0 - S_{1}\\&+\sum_{k=p}^{n-1}c_k{k\choose p}((1+As)\pi^*\chi)^{n-k}((1+As)\pi^*\Omega_1)^{k-p},
 \end{align*}
where $S_0$ are the terms in the triple summation with $a = 0$ (and hence is a double summation), while $S_1$ is the triple summation with  $a\geq 1$. By changing variable $b = n-q$, ans switching the order of the summations, we can rewrite $$S_0 =  \sum_{q=p+1}^n\Big(\sum_{k=n-(q-p)}^{n-1}c_k{k\choose p}{k-p\choose n-q}A^{k-n}\Big)(1+As)^{n-q}(\pi^*\Omega_1)^{n-q}(As^2\tilde\Omega_0)^{q-p}.$$ Now choose $A>>1$, such that for $q = p+1,\cdots,n$, $$\sum_{k=n-(q-p)}^{n-1}c_k{k\choose p}{k-p\choose n-q}A^{k-n} \leq \frac{\vep}{4}{n\choose p}{n-p\choose n-q}.$$ Note that this can be done since $k-n\leq -1$. Combining with the above, we then obtain 
\begin{align*}
L-R&\geq {n\choose p}(1-\frac{\vep}{2})\sum_{q = p+1}^{n}{n-p\choose n-q}(1+As)^{n-q}\pi^*\Omega_1^{n-q}(As^2\tilde\Omega_0)^{q-p} - S_{1}\\&+\sum_{k=p}^{n-1}c_k{k\choose p}((1+As)\pi^*\chi)^{n-k}((1+As)\pi^*\Omega_1)^{k-p}.
\end{align*}
We now focus on $S_1$. We change indices: $b = l-q$ and $a = n-l$. Then 
\begin{align*}
S_1 &=  \sum_{k=p}^{n-1}\sum_{l=k}^{n-1}\sum_{q=l-(k-p)}^{l}c_k{k\choose p}{n-k\choose n-l}{k-p\choose l-q}A^{k-l}(1+As)^{n-q}(\pi^*\chi)^{n-l}(\pi^*\Omega_1)^{l-q}(As^2\tilde\Omega_0)^{q-p} \\
&=  \sum_{q=p}^{n-1}\sum_{l=q}^{n-1}\sum_{k=l-(q-p)}^{l}c_k{k\choose p}{n-k\choose n-l}{k-p\choose l-q}A^{k-l}(1+As)^{n-q}(\pi^*\chi)^{n-l}(\pi^*\Omega_1)^{l-q}(As^2\tilde\Omega_0)^{q-p} \\
\end{align*}
Let $Q_0$ be the term when $k=l$ in  the innermost summation, and $Q_1$ be the rest. One can choose $A>>1$ (independent of $s$) such that for each $q = p,\cdots,n-1$, and for each $l=q,\cdots,n-1$, 
\begin{align*}
\sum_{k=l-(q-p)}^{l-1}c_k{k\choose p}{n-k\choose n-l}{k-p\choose n-q}A^{k-l}(1+As)^{n-q}(\pi^*\chi)^{n-l}(\pi^*\Omega_1)^{l-q}(As^2\tilde\Omega_0)^{q-p}\\
= A^{-1}\sum_{k=l-(q-p)}^{l-1}c_k{k\choose p}{n-k\choose n-l}{k-p\choose n-q}A^{k-l+1}(1+As)^{n-q}(\pi^*\chi)^{n-l}(\pi^*\Omega_1)^{l-q}(As^2\tilde\Omega_0)^{q-p}\\
\leq \frac{\vep}{4n^2} {n\choose p} {n-p\choose n-q}(1+As)^{n-q}\pi^*\Omega_1^{n-q}(As^2\tilde\Omega_0)^{q-p},
\end{align*}
so that summing over $q$ and $l$, we have that $$Q_1 \leq {n\choose p}\frac{\vep}{4}\sum_{q = p+1}^{n}{n-p\choose n-q}(1+As)^{n-q}\pi^*\Omega_1^{n-q}(As^2\tilde\Omega_0)^{q-p},$$ and hence 
\begin{align*}
L-R &\geq {n\choose p}(1-\frac{3\vep}{4})\sum_{q = p+1}^{n}{n-p\choose n-q}(1+As)^{n-q}\pi^*\Omega_1^{n-q}(As^2\tilde\Omega_0)^{q-p} - Q_0\\
&+\sum_{l=p}^{n-1}c_l{l\choose p}((1+As)\pi^*\chi)^{n-l}((1+As)\pi^*\Omega_1)^{l-p}
\end{align*}
Note that $${n\choose q}{l\choose p}{l-p\choose l-q} = {n\choose p}{n-p\choose n-q}{l\choose q},$$ and so  
\begin{align*}
Q_0&-\sum_{l=p}^{n-1}c_k{l\choose p}((1+As)\pi^*\chi)^{n-l}((1+As)\pi^*\Omega_1)^{l-p} \\
&=   \sum_{q=p+1}^{n-1}\sum_{l=q}^{n-1}c_l{l\choose p}{l-p\choose l-q}(1+As)^{n-q}(\pi^*\chi)^{n-l}(\pi^*\Omega_1)^{l-q}(As^2\tilde\Omega_0)^{q-p} \\
&\leq  \sum_{q=p+1}^{n}\sum_{l=q}^{n-1}c_l{l\choose p}{l-p\choose l-q}(1+As)^{n-q}(\pi^*\chi)^{n-l}(\pi^*\Omega_1)^{l-q}(As^2\tilde\Omega_0)^{q-p}\\
&= \sum_{q=p+1}^n\frac{{n-p\choose n-q}{n\choose p}}{{n\choose q}}\Big(\sum_{l=q}^{n-1}c_l{l\choose q}(\pi^*\chi)^{n-l}(\pi^*\Omega_1)^{l-q}\Big)(1+As)^{n-q}(As^2\tilde\Omega_0)^{q-p}\\
&\leq {n\choose p}\Big(1-\frac{3\vep}{4}\Big)\sum_{q=p+1}^n{n-p\choose n-q}(\pi^*\Omega_1)^{n-q}(1+As)^{n-q}(As^2\tilde\Omega_0)^{q-p}.
\end{align*}
To recap, for any $x\in \pi^{-1}(U')$, we have verified that \eqref{claim2}, and hence \eqref{claim}, holds. 

To complete the proof, we consider the points in the set $\tilde M\setminus \pi^{-1}(U')$. We now go back to working with $\tilde\chi_s$ and prove the inequality in \eqref{claim} directly. On this set, $\pi$ is a bi-holomorphism and hence for some uniform $C$ (independent of $s$, so long as $s<1$),
\begin{align*}
C^{-1}\pi^*\chi \leq  \pi^*\Omega_1 \leq C\pi^*\chi\\
C^{-1}\pi^*\chi \leq \tilde\Omega_0\leq C\pi^{*}\chi.
\end{align*}
Consider the difference of the two sides in \eqref{claim}. The $O(1)$ terms (as $s\rightarrow 0$) cancel out, and the leading term (or the $O(s)$ term) is given by s $$As\sum_{k=p}^{n-1}c_k(n-k){k\choose p}(\pi^*\chi)^{n-k}((1+As)\pi^*\Omega_1)^{k-p},$$ which is strictly positive since at least one of $c_k$ is strictly positive and others are non-negative. The rest of the terms are $O(s^2)$ (since they contain at least one copy of $s^2\tilde\Omega_0$), and hence for {\em all} $s<s_0$, where $s_0$ only depends on the constant $C$ above, the difference of the two sides in \eqref{claim} is positive.  
\end{proof}

\section{Proof of Theorem \ref{thm:main2}}\label{secondmaintheorem}
\indent In this section, we prove Theorem \ref{thm:main2}. In particular, we assume that the (non-uniform) numerical condition holds, i.e., for any subvariety $V \subset M$,
\begin{equation}\label{eq:stability}
\int_V\Big({n\choose p}[\Omega_0]^{n-p} - \sum_{k=p}^{n-1}c_k{k\choose p}[\chi^{n-k}]\cdot [\Omega_0^{k-p}]\Big) > 0.
\end{equation}
We prove Theorem \ref{thm:main2} by induction on $n=dim(M)$ just as in the case of Theorem \ref{thm:main}. For $n=1$, the result is trivial. Assume that it holds for $1,2,\ldots, n-1$. Proposition \ref{prop:gluing} shows that the proof of Theorem \ref{thm:main} carries over verbatim provided we  prove the following stronger version of Lemma \ref{localmetric}.
\begin{proposition}\label{prop:nonuniformlocal}
Let $Z \subset M$ be a closed analytic subset. Assume the (non-uniform) numerical condition \eqref{eq:stability} holds. Then there exists a neighbourhood $U$ of $Z$ and a  K\"ahler metric $\Omega_U = \Omega_0 + \ddbar\psi_U $ on $U$ such that $$n\Omega_U^{n-1} - \sum_{k=1}^{n-1} kc_k  \chi^{n-k} \Omega_U^{k-1} >0.$$
\end{proposition} 
To prove Proposition \ref{prop:nonuniformlocal}, we first prove a lemma that is similar to, but stronger than Lemma \ref{lem:local}.
\begin{lemma}\label{lem:nonuniformlocal}
Let $Z\subset M$ be a smooth $m$-dimensional subvariety.  Let $\Omega_0$ and $\chi$ be K\"ahler forms on $M$. Suppose there exists a smooth K\"ahler form $\omega_Z = \Omega_0\Big|_{Z} + \ddbar\psi_Z$ on $Z$ such that $$\omega_Z^{m} - \sum_{j= 0}^{m-1}b_j\chi^{m-j}\omega_Z^{j}>0.$$   \\
\indent Then there exists a neighbourhood $U$ of $Z$, a smooth function $\psi_U :U\rightarrow \mathbb{R}$ such that the smooth form $\Omega_U = \Omega_0 + \ddbar\psi_U $ is K\"ahler and satisfies  $$n\Omega_{U}^{n-1} - \sum_{k=1}^{n-1} k c_k  \chi^{n-k} \Omega_{U}^{k-1} >0$$ on $U$. In particular, $\Omega_U$ satisfies the cone condition on $U$.
\end{lemma}
\begin{proof}
Let $dist(.,Z)$ be the distance (with respect to some K\"ahler metric on $M$) to $Z$ and $\psi'$ the extension of $\psi$ (as in the proof of Lemma \ref{lem:local}) to a neighbourhood of $Z$. Then, for a large enough constant $C$, $\tilde{\psi} =\psi' + Cdist(.,Z)^2$ is a smooth function in a neighbourhood $U$ of $Z$ satisfying the desired conditions.  
\end{proof}
Using this lemma, the induction hypothesis, and the proof of Lemma \ref{lem:localconesmooth}, we see that the following non-uniform version of Lemma \ref{lem:localconesmooth} holds.
\begin{lemma}\label{lem:nonuniformlocalconesmooth}
Let $Z\subset M$ be a smooth sub-variety  of dimension $m<n$. Assuming the (non-uniform) numerical condition \eqref{eq:stability}, there exists a neighbourhood $U$ of $Z$ and a K\"ahler metric $\Omega_U = \Omega_0 + \ddbar\psi_U$ on $U$ such that $$n\Omega_{U}^{n-1} - \sum_{k=1}^{n-1} c_k k \chi^{n-k} \Omega_{U}^{k-1} >0.$$ 
\end{lemma}
We only need to observe that the quantity $a_Z$ in the proof of Lemma \ref{lem:localconesmooth} is still positive, and that we can apply the same argument with $\vep_Z = a_Z/2$. Now we proceed to prove Proposition \ref{prop:nonuniformlocal}.

\subsection{Proof of Proposition \ref{prop:nonuniformlocal}}  We induct on the maximal dimension $m<n$ over all irreducible components of $Z$. For $m=0$, we have isolated points, for which the proposition is trivially true. Assume it is true for $0,1,2\ldots, m-1$. Removing the isolated points of $Z$, without loss of generality, we may assume that the dimension of every irreducible component of $Z$ is $\geq 1$. \\
\indent As in the proof of Theorem \ref{thm:main}, we resolve $Z$ to a disjoint union of smooth manifolds $\tilde{Z}=Z_r$ via a sequence of $r$ blowups $M_1, M_2\ldots$ at smooth  centres $C_i \subset Z_i$ where $Z_i$ are proper transforms. Let $\pi :\tilde{Z} \rightarrow Z$ be the projection map with $\pi^{-1}(Z_{sing}) = E$ as the union of all the exceptional divisors $E_i$. There exists an ample line bundle $L$ over $M$. For some large $N$, the bundle $L^N$ has a holomorphic section $\mathcal{S}_Y$ whose zero locus $Y$ is a smooth connected hypersurface such that $Y_1=Y\cap Z$ is a divisor in $Z$\footnote{Indeed, if $dim(Z)=1$, then for large $N$, $Y \cap Z$ is a non-empty collection of points and is hence a divisor in $Z$. If $dim(Z)\geq 2$, then since $dim(Y)+dim(Z)>dim(M)$ the analytic set $Y\cap Z$ is a union of connected subvarieties by the Fulton-Hansen connectedness theorem. Hence, if we ensure that there exists a $Y$ that intersects every component of $Z$ in at least one smooth point transversally, we are done by Bertini's theorem.}.  Let $N$ also be so large that $\tilde{L}_i=\pi^* L^N\otimes [-E_i] \otimes [-E_{i-1}]\ldots$ is an ample line bundle over $M_i$ for all $1\leq i \leq r$. Let $h$ be a metric on $L$ (over $M$) of positive curvature $F_h$ and $h_i$ be metrics on $[-E_i]$ such that $F_{h_i} \geq -\frac{N}{10r}\pi_i^* F_h$ and is positive when restricted to $E_i$. Then $H_i = \pi_i^*h^N \otimes \Pi_j h_j$ is a metric on $\tilde{L}_i$ of positive curvature $F_{H_i}$. Denote by $ \tilde\Omega_{0}$ the K\"ahler form $\pi^* \Omega_0 + C^{-1}F_{H_r}$ for a sufficiently large $C$. 

Now define $\tilde\chi_{s}:=\pi^*\chi+s^{5n} \tilde\Omega_{0}$ for $1\geq s\geq 0$, and
\begin{gather}
\tilde{a}_{s,s_0} = \frac{\int_{\tilde{M}} (\tilde\chi_s^n -\tilde\chi_{s_0}^n)}{\int_{\tilde{M}}\pi^*\chi^n}.
\label{conditionontildeas0}
\end{gather}
Note that $\tilde \chi_s \leq C_{\chi} \tilde\Omega_0$ for some $C_{\chi}>0$ independent of $0\leq s \leq 1$. We fix $\tilde{\varepsilon}$ and $s_0<1$ for the remainder of this paper so that $\tilde{\chi}_{s_0}$ is K\"ahler for all $0<s<s_0$ and the following hold for all $0<s<s_0$ and $0\leq t\leq 1$.
\begin{gather}
200\tilde{a}_{s,s_0} >\frac{f_m}{4^n}, \ \mathrm{and} \nonumber\\
\tilde{\chi}_{s_0}-\tilde{\chi}_{\frac{s_0}{2}} -\tilde{\varepsilon}N\frac{\pi^*F_h \vert \pi^* \mathcal{S}_Y \vert_{h^N}^2}{\vert \pi^* \mathcal{S}_Y \vert_{h^N}^2+t^2} \geq 0,
\label{s0fixed}
\end{gather}
where $f_m$ is as in \ref{thm:main}.\\

Our first step is to state and prove another concentration of mass result (like Theorem \ref{thm:conc-mass}). We need the following definitions, which generalise Definition \ref{degenconedef} in two directions, namely allowing powers of arbitrary order, and allowing $\chi$ to be possibly only semi-positive.

\begin{definition}\label{degenconedef-m-power}
Let $\Theta$ be a closed, positive $(1,1)$ current and $\chi$  a K\"ahler form on an $n$-dimensional K\"ahler manifold. We say that $\Theta$ satisfies the \emph{($\vep,n-p$)-uniform cone condition} ($0\leq p\leq n$):
\begin{equation}\label{singcone}
{n\choose p}(1-\vep)\Theta^{n-p} - \sum_{k=p}^{n-1} c_k {k\choose p} \chi^{n-k} \Theta^{k-p} \geq 0,
\end{equation}
if for any coordinate chart $U$ with $\Theta\Big|_U = \ddbar\varphi_U$, on $U_\delta:=\{x\in M~|~ B(x,\delta)\subset U\}$ (where $B(x,\delta)$ is a coordinate Euclidean ball of radius $\delta$ centred at $x$), we have $${n\choose p}(1-\vep)(\ddbar\varphi_{U,\delta})^{n-p} (x) - \sum_{k=p}^{n-1} c_k k \chi_0^{n-k} (\ddbar\varphi_{U,\delta})^{k-p} (x) \geq 0,$$  for any K\"ahler metric $\chi_0$ on $U$ with constant coefficients satisfying $\chi_0\leq \chi$ on $B(x,\delta)$. Here $\varphi_{U,\delta}$ are the mollifications of $\varphi_U$ as before. \\
\end{definition}
Note that in particular, Definition \ref{degenconedef} defines an {($\vep,n-1$)-uniform cone condition}.

\begin{definition}\label{def:strongcone}
 Let $\Theta$ be a closed, positive $(1,1)$-current and $\chi$ be a semi-positive smooth closed $(1,1)$-form on a K\"ahler manifold of dimension $n$. We say that $\Theta$ satisfies the $(n-p)$-degenerate strong cone condition 
\begin{equation}\label{singcone}
{n\choose p}\Theta^{n-p} - \sum_{k=1}^{n} c_k{k\choose p} \chi^{n-k} \Theta^{k} \geq 0,
\end{equation}
if there exist positive sequences $\vep_i,\mu_i, \nu_i, \gamma_i \rightarrow 0$ and a smooth K\"ahler form $\tilde{\Omega}_0$, such that the closed positive current $\Theta_i=\Theta (1+\gamma_i)+\mu_i\tilde{\Omega}_0$ satisfies the $(\vep_i,n)$-uniform cone condition  with respect to the K\"ahler form $\chi_i=\chi+\nu_i \tilde{\Omega}_0$ for every $i$.  
\end{definition}
\begin{remark}
It is easy to see that for a \emph{smooth} current $\Theta$, Definition \ref{def:strongcone} is equivalent to the usual pointwise definition. It is also not hard to see that Definition \ref{def:strongcone} generalises Definition \ref{degenconedef}.
\end{remark}
Recall that $Y\subset M$ is an ample divisor such that $Y$ intersects every irreducible component of $Z$ in a divisor. Define $\tilde{Y}=\pi^{-1}Y \cap \tilde{Z}$ and $\tilde{E}=E\cap \tilde{Z}$. Note that $\tilde{Y}$ and $\tilde{E}$ are divisors in $\tilde{Z}$. In fact, $E$ intersects $\tilde{Z}$ in simple normal crossings. Denote $\pi^*\chi$ by $\tilde{\chi}$. Once again, as in equation \eqref{eq:b}, for  $m<n$ and $j = 0,\cdots,m-1$,we let\begin{equation}\label{eq:b2}
b_j := \frac{c_{j+n-m}{j+n-m\choose n-m}}{{n\choose m}}.
\end{equation}

\begin{proposition}\label{prop:conc-massnonuniform}
 There exists a constant $\beta'>0$ and a current $\Theta \in [\pi^* \Omega_0]$ on $\tilde{Z}$ such that $\Theta \geq 2\beta' [\tilde{Y}]$ and $\Theta$ satisfies the following $m$-degenerate strong cone condition  (in the sense of Definition \ref{def:strongcone}):
$$\Theta^{m} - \sum_{j=0}^{m-1} b_j  \pi^*{\chi}^{m-j} \Theta^{j} \geq 0.$$
\end{proposition}

To this end, we first need the following general lemma regarding the cone condition. This is a stronger (i.e. degenerate) version of Lemma \ref{lem:cone-condition-auxilary}. 
\begin{lemma}
Let $\tilde Z$ be an $m$-dimensional K\"ahler submanifold of an $n$-dimensional K\"ahler manifold, and for $0\leq j \leq m-1$, let $b_j \geq 0$ such that $\sum_j b_j>0$. Furthermore, let $\Omega, \chi \geq 0$ be smooth forms such  that  for all $0\leq p \leq m-1$, the following point-wise inequality is satisfied at a point $x\in \tilde{Z}$: \begin{gather}
\frac{1}{(1+2\varepsilon)^2}  \Omega^{m-p}(x) - \displaystyle \sum_{j=p}^{m-1} \frac{{j \choose p}}{{m \choose p}} b_j \chi^{m-j} \Omega^{j-p} (x)\geq 0.
\label{eq:generaldegencone}
\end{gather}
Assume that $ \chi(x) \leq C_{\chi}  \tilde\Omega_{0} (x)$ for some $C_{\chi}>0$. Then there exists an $s'_{0}=s'_0(C_{\chi}, m, b_j, \epsilon)$  such that for all $s<s'_0$, the following point-wise inequality holds at $x$:
\begin{align}\label{eq:jackedupconecondition}
\frac{1}{(1+\varepsilon)^2}(\Omega(1+10s) + 10s^2  \tilde\Omega_{0})^m (x)- \displaystyle \sum_{j=0}^{m-1} b_j (\chi+s^{5n}  \tilde\Omega_{0})^{m-j} (\Omega(1+10s) + 10s^2  \tilde\Omega_{0})^j (x) \geq 0,
\end{align}
\label{lem:jackinguptheconecondition}
\end{lemma}
\begin{proof}
The left hand side of Inequality \ref{eq:jackedupconecondition} can be written as $P+Q$, where :
\begin{align}
P &:= \frac{1}{(1+2\varepsilon)^2}(\Omega(1+10s) + 10s^2  \tilde\Omega_{0})^m - \displaystyle \sum_{j=0}^{m-1} b_j \chi^{m-j} (\Omega(1+10s) + 10s^2  \tilde\Omega_{0})^j \nonumber \\
Q &:= \left (\frac{1}{(1+\varepsilon)^2}-\frac{1}{(1+2\varepsilon)^2}\right )(\Omega(1+10s) + 10s^2  \tilde\Omega_{0})^m \nonumber \\
&- \displaystyle \sum_{j=0}^{m-1}\sum_{h=0}^{m-j-1} b_j {m-j \choose h} \chi^h s^{5n(m-j-h)} \tilde\Omega_{0} ^{m-j-h}(\Omega(1+10s) + 10s^2  \tilde\Omega_{0})^j.
\label{eq:splittingupinjackingup}
\end{align}
We first prove that $P\geq 0$. In the following calculation, by definition ${n \choose r}=0$ when $r>n$.
\begin{align*}
P &= \sum_{p=0}^m (10s^2)^p  \tilde\Omega_{0} ^p\frac{1}{(1+2\varepsilon)^2} {m \choose p} \Omega^{m-p} (1+10s)^{m-p} - \displaystyle \sum_{j=0}^{m-1} b_j \chi^{m-j} \sum_{p=0}^j (10s^2)^p  \tilde\Omega_{0} ^p {j \choose p} \Omega^{j-p} (1+10s)^{j-p} \nonumber \\
&\geq \sum_{p=0}^m (10s^2)^p (1+10s)^{m-p}  \tilde\Omega_{0} ^p\Bigg ( \frac{1}{(1+2\varepsilon)^2} {m \choose p} \Omega^{m-p}  - \displaystyle \sum_{j=0}^{m-1} b_j \chi^{m-j} {j \choose p} \Omega^{j-p} \Bigg ) \nonumber \\
&\geq 0,
\end{align*}
using Inequalities \ref{eq:generaldegencone}.\\
Next, we prove that $Q\geq 0$, thus completing the proof of Lemma \ref{lem:jackinguptheconecondition}. To begin with, since $\chi\leq C_{\chi}\tilde{\Omega}_0$, and $5m\leq 5n(m-j-h)$ for all $0\leq h \leq m-j-1$, assuming that $s<1$ we get the following inequality. 
\begin{align}
Q &\geq \left (\frac{1}{(1+\varepsilon)^2}-\frac{1}{(1+2\varepsilon)^2}\right )(\Omega(1+10s) + 10s^2  \tilde\Omega_{0})^m \nonumber \\
&- \displaystyle \sum_{j=0}^{m-1}\sum_{h=0}^{m-j-1} b_j {m-j \choose h} C_{\chi}^h  s^{5m} \tilde\Omega_{0} ^{m-j}(\Omega(1+10s) + 10s^2  \tilde\Omega_{0})^j. \nonumber 
\end{align}
We continue simplifying further to get the following (using $s<1$).
\begin{align}
\Rightarrow Q &\geq \sum_{j=0}^{m-1} (\Omega(1+10s) + 10s^2  \tilde\Omega_{0})^j (s^2  \tilde\Omega_{0})^{m-j} \Bigg ( \left (\frac{1}{(1+\varepsilon)^2}-\frac{1}{(1+2\varepsilon)^2}\right )10^{m-j} \nonumber \\
&-  \sum_{h=0}^{m-j-1} b_j {m-j \choose h} C_{\chi}^h  s^{3m} \Bigg ) \nonumber \\
&\geq 0,
\end{align}
for sufficiently small $s$ (depending only on $C_{\chi}, m, b_j, \epsilon$).
\end{proof}
Now we prove Proposition \ref{prop:conc-massnonuniform}.

\begin{proof}
For $0<t<1$, the K\"ahler classes $[\Omega_t] := (1+t)[\Omega_0]$ satisfy the uniform numerical criteria (with the $\vep$ possibly depending on $t$). Then by Lemma \ref{num-resolution}, there exists $A_t>0$ and $\frac{s_0}{2}>\overline{s}_t>0$ such that the classes $(1+A_ts)[\pi^*\Omega_t] + A_ts^2[\tilde\Omega_0]$ and $[\tilde\chi_{s,t}] := [\pi^*\chi] + s^2[\tilde\Omega_0]$ satisfies the uniform numerical criteria for all $s<\overline{s}_t<\frac{s_0}{2}<1$. In particular, the uniform numerical criteria are satisfied with $[\tilde\chi_{s,t}]$ replaced by $[\tilde\chi_s]$. Now let $s_t := \frac{\overline{s}_t (1-e^{-t})}{2A_t}$, and $[\tilde\Omega_{s_t}] := (1+A_ts_t)[\pi^*\Omega_t] + A_ts_t^2[\tilde\Omega_0]$. Then $[\tilde\Omega_{s_t}]$ and $[\tilde\chi_{s_t}]$ satisfy the uniform numerical criteria (where once again the uniformity might depend on $t$). Consider $$\mathcal{B}_{s_t} := \frac{\int_{\tilde M}\Big([\tilde\Omega_{s_t}]^n  - \sum_{j=1}^{n-1}c_j\tilde\chi_{s_t}^{n-j}[\tilde\Omega_{s_t}]^j - \tilde \chi_{s_0}^n +\tilde \chi_{s_t}^n\Big)}{\int_{\tilde M}\tilde\chi_{s_t}^n}.$$ 
Note that $\mathcal{B}_{s_t}$ can be negative but since by assumption $s_t<\frac{s_0}{2}$ for all $0<t<1$, we see that 
\begin{gather}
\mathcal{B}_{s_t}\geq a_{s_t,s_0}>\frac{f_m}{4^n200},
\label{ineq:onb}
\end{gather}
 where $f_m$ is as in \eqref{eqn:f}. \\
\indent Next, as before, we construct some model forms that concentrate on $\tilde Y$.  We let $$\tilde \eta_{t} = \tilde \chi_{s_0} + \tilde{\varepsilon}\ddbar\log(|\pi^*S|^2_{\pi^*h^N} + t^2).$$ We then find a metric $\tilde\Omega_{s_t}\in [\tilde\Omega_{s_t}]$  on $\tilde{M}$ solving 
\begin{gather}
\tilde\Omega_{s_t}^n = \sum_{j=1}^{n-1}c_j\tilde\chi_{s_t}^{n-j}\tilde\Omega_{s_t}^j + f_{t}\tilde\chi_{s_t}^n,
\label{eqn:massconc2}
\end{gather}
 where $f_{t} := \frac{\tilde\eta_{t}^n}{\tilde\chi_{s_t}^n} - 1+\mathcal{B}_{s_t}.$ Indeed, noting that $$\displaystyle \int_{\tilde{M}} \tilde\Omega_{s_t}^n \leq \int_{\tilde{M}}2^n((1+t)^n\pi^*\Omega_0^n (1+A_t s_t)^n + A_t^n s_t^{2n}\tilde\Omega_0^n),$$ and using Inequality \ref{ineq:onb}, we see that $f_t$ satisfies the conditions in Theorem \ref{thm:main} for all $0<t<<1$.\\
\indent Moreover, since the uniform positivity criteria are satisfied for every $0<t<1$, by Theorem \ref{thm:main}, there exists a smooth solution of \ref{eqn:massconc2} for all $0<t<<1$. By Lemma \ref{necessity} we see that the cone condition is met by $\tilde{\Omega}_{s_t}$ with respect to $\tilde{\chi}_{s_t}$. Choose a positive sequence $\gamma_i\rightarrow 0$. There exists a sequence $\delta_i>0$ such that the $(2\delta_i,n)$-uniform cone condition is satisfied by $\tilde{\Omega}_{s_t}(1+\gamma_i)$ with respect to $\tilde{\chi}_{s_t}$.  Lemma \ref{lem:jackinguptheconecondition} shows that there exists a positive sequence $\mu_i$ converging to $0$ such that $\vep_i$-cone condition is satisfied by $\tilde{\Omega}_{s_t}(1+\gamma_i)(1+10\mu_i)+10\mu_i^2 \tilde{\Omega}_0$ with respect to $\tilde{\chi}_{s_t}+\mu_i^{5n}\tilde{\Omega}_0$ for some $\vep_i$ depending explicitly on $\delta_i$. Denote a weak limit of $\tilde{\Omega}_{s_t}$ by $\Theta$. Then  arguing just as in the proof of Theorem \ref{thm:conc-mass}, it is easily seen that $\Theta(1+\gamma_i)(1+10\mu_i)+10\mu_i^2 \tilde{\Omega}_0$ satisfies the $(\vep_i,m)$-uniform cone condition with respect to the K\"ahler form $\pi^* \chi+\mu_i^{5n}\tilde{\Omega}_0$ as in Definition \ref{degenconedef-m-power}, and hence $\Theta$ satisfies the $m$-degenerate strong cone condition as in Definition \ref{def:strongcone}.  \\   
\indent To complete the proof of the Proposition, it is then enough to prove the following analogue of Lemma \ref{pos-mass} to prove that any weak limit $\Theta$ of $\tilde{\Omega}_{s_t} \vert_Z$ satisfies $\Theta \geq 2 \beta ' [\tilde{Y}]$.  
\begin{lemma}\label{pos-mass-deg}
For any $y\in \tilde Y$, there exists a neighbourhood $U$ and a constant $\be_U>0$ such that for all $t<<t_0$, $$\int_{U\cap V_t}\tilde\Omega_{s_t}\wedge\tilde\chi_{s_0}^{m-1}>\be_U,$$ where as before, $V_t=\{|S|_{h^N}<t\}$.
\end{lemma}
The proof proceeds exactly as in Lemma \ref{pos-mass}. The key observation is that $\tilde\chi_{s_0}$ is K\"ahler throughout $\tilde{Z}$, and hence the argument in proof of Lemma 2.1 in \cite{dp} implies that for each $y\in \tilde Y$ there exists a neighbourhood $U$ and a constant $\be(U)$ such that $$\int_{U\cap V_t}\tilde\eta_t\wedge\tilde\chi_{s_0}^{m-1}>\be(U).$$
Now, since $\tilde{\Omega}_{s_t}$ satisfies $$n\tilde{\Omega}_{s_t}^{n-1}-\sum_{k=1}^{n-1} kc_k\tilde{\chi}_{s_t}^{n-k}\tilde{\Omega}_{s_t}^{k-1}\geq 0,$$
we see that $$\tilde{\Omega}_{s_t}^{m} - \sum_{j=0}^{m-1} b_j  \tilde{\chi}_{s_t}^{m-j} \tilde{\Omega}_{s_t}^{j} \geq 0$$ when restricted to $\tilde{Z}$.\\  

\end{proof}

Akin to the proof of Proposition \ref{prop:conc-massnonuniform}, denote by $\tilde{\Omega}_s$ the form $\pi^* \Omega_0 (1+10s)+10s^2 \tilde\Omega_0$. We now use Proposition \ref{prop:conc-massnonuniform} to prove a similar proposition (with more positivity in some sense).
\begin{proposition}\label{prop:conc-massnonuniform-nondeg}
There exist constants $\beta''>0, K>1, \epsilon >0$ and $s_0'<s_0/2$ such that for all $s<s'_{0}$, there exists a K\"ahler current $T_{s} \in [\tilde{\Omega}_{s}]$ on $\tilde{Z}$ with the following properties. 
\begin{enumerate}
\item $T'_{s}=T_{s}-250\epsilon \tilde\Omega_0 \geq \beta' [\tilde{Y}] + \beta''[\tilde{E}] + \frac{\tilde\Omega_0}{K}$. In particular, $T_s$ and $T_s'$ have Lelong numbers at least $\min(\be',\be'')$ along $\tilde Y \cup \tilde E$.
\item  For all $c>0$, the Lelong sublevel set $E_{c}(T_{s}) \subset E_{\frac{c}{3}}(\Theta)\cup \tilde{Y}\cup \tilde{E}$. 
\item $T'_{s}$ satisfies the following degenerate cone condition  (in the sense of Definition \ref{degenconedef}) away from $\tilde{Y}\cup \tilde{E}$:
$$\frac{1}{(1+\epsilon)^2}(T'_{s})^{m} - \sum_{j=0}^{m-1} b_j  \tilde{\chi}_s^{m-j} (T'_{s})^{j} \geq 0.$$
\end{enumerate}
\end{proposition}

\begin{proof} [Proof of Proposition \ref{prop:conc-massnonuniform-nondeg}]
 Let $s_E$ be an induced section over $\tilde{M}$ of $[E]$ defining $E$. For small enough $\beta''$, we can easily see that the current $T_0=\Theta-\beta'\spbp \ln \vert \pi^* \mathcal{S}_Y \vert_{h^N} ^2 +\beta'' \spbp \ln \vert s_E \vert^2_{\otimes h_i^{-1}}$ is K\"ahler throughout $\tilde{Z}$ and satisfies $T_0'=T_0-250\epsilon_0 \tilde\Omega_0 \geq \beta'[\tilde{Y}]+\beta''[\tilde{E}]+\frac{\tilde\Omega_0}{K}$ for constants $\epsilon_0>0$ and  $K>1$. In fact, going back to $\tilde{\Omega}_{s_t}$ in the proof of Proposition \ref{prop:conc-massnonuniform}, we see that $$\tilde{\Omega}'_{t}:=\tilde{\Omega}_{s_t}-250\epsilon\tilde\Omega_0 -\beta'\spbp \ln \vert \pi^* \mathcal{S}_Y \vert_{h^N} ^2 +\beta'' \spbp \ln \vert s_E \vert^2_{\otimes h_i^{-1}} \geq \tilde\Omega_{s_t} +\frac{\tilde\Omega_0}{K}$$ for all $\epsilon<\epsilon_0$ and all $0<t<1$ away from $\tilde{Y}$. Using Lemma \ref{lem:strictcone}, we see that $\tilde\Omega_t'$ satisfies the following cone condition pointwise  on $\tilde Z \setminus (\tilde{Y}\cup \tilde{E})$.
\begin{gather}
\frac{1}{(1+4\epsilon)^2}(\tilde{\Omega}'_{t})^m-\sum_{j=0}^{m-1}b_j \tilde\chi_{s_t}^{m-j} (\tilde{\Omega}'_{t})^j \geq 0,
\end{gather}
for some uniform choice of $\epsilon<\epsilon_0$ and all $0<t<<1$. Since $\tilde\chi_{s_t} \leq C_{\chi}  \tilde\Omega_{0}$ for a uniform $C_{\chi}$ (independent of $t$) for all points of $\tilde{M}$, we can apply Lemma \ref{lem:jackinguptheconecondition} (and monotonicity in $\Omega$) to conclude that there exists $s'_{0}$ such that for $s<s'_{0}$, 
\begin{gather}
\frac{1}{(1+2\epsilon)^2}(\tilde{\Omega}''_{t,s})^m-\sum_{j=0}^{m-1}b_j \tilde\chi_{s_t,s}^{m-j} (\tilde{\Omega}''_{t,s})^j \geq 0,
\end{gather}
where $\tilde{\Omega}''_{t,s} = \tilde{\Omega}'_{t}(1+10s)+10s^2  \tilde\Omega_{0}$ and $\tilde\chi_{s_t,s} = \tilde{\chi}_{s_t}+s^{5n} \tilde\Omega_{0}$. \\
\indent A weak limit of  $\tilde{\Omega}''_{t,s}$ is $T'_{s}=T_0'(1+10s)+10s^2  \tilde\Omega_{0}$. Hence, its Lelong sublevel set $E_c(T'_{s})$ is contained in $E_{\frac{c}{1+10s}}(\Theta)\cup \tilde{Y} \cup \tilde{E}$. By shrinking $s_0'$, we may assume that for all $s<s_0'$, $1+10s<\frac{3}{2}$ and so $E_c(T'_{s}) \subset E_{c/3} (\Theta) \cup \tilde{Y} \cup \tilde{E}$. By the reasoning in the proof of Theorem \ref{thm:conc-mass}, we see that $T'_{s}$ satisfies the desired degenerate cone condition. 
\end{proof}
\indent Let $c>0$ (independent of $s$) be a constant such that $T_{s}$ has a Lelong number $\geq 100c$ on $\tilde{Y}\cup \tilde{E}$. We now proceed to regularise $T_{s}$ using the regularised maximum construction and glue it with a smooth metric (that we produce using the induction hypothesis) in an appropriately chosen neighbourhood of $\tilde{Y} \cup \tilde{E}$. As before, cover $\tilde{Z}$ with $\tilde\Omega_0$-balls $B_{2r}(x_i)$ of radius $2r<1$ such that $B_{r}^i=B_{r}(x_i)$ also cover it and the following hold.
\begin{enumerate}
\item For every $z\in \tilde{Y} \cup \tilde{E}$, there exists an $i$ such that $B_{r/4}(z)\subset B_{r}^i$ where $B_{r/4}(z)$ is the $\tilde\Omega_0$-ball of radius $\frac{r}{4}$ around $z$.
\item On $B^i_{2r}$, there exist holomorphic coordinates such that $\tilde\Omega_0 = \spbp \phi_{\tilde\Omega_0}^i$, and
\begin{gather}
\vert \phi_{\tilde\Omega_{0}}^i - \vert z \vert^2 \vert < \epsilon r^2, \nonumber \\
\Big (1-\frac{\epsilon}{200} \Big )\spbp \vert z \vert^2 \leq \tilde\Omega_{0} \leq \Big (1+\frac{\epsilon}{200} \Big )\spbp \vert z\vert^2.
\label{conditionsoncoordcharts} 
\end{gather}
\end{enumerate}
Define $$\epsilon_{\ref{secondmaintheorem}}=\min \left (c , \frac{\epsilon c_n r^2}{40}\right ),$$ where as before, $c_n$ is given by \eqref{eq:cn}. Note that this constant is independent of $s$.\\
\indent  The subset $S=\pi(E_{\epsilon_{\ref{secondmaintheorem}}/5}(\Theta)\cup \tilde{Y} \cup \tilde{E})\subset M$ is an analytic subset of maximal dimension $<m$. Hence, by the induction hypothesis, in a neighbourhood $U_{S}\subset M$ of $S$, there exists a constant $0<\epsilon_{1}<< \epsilon$ and a smooth K\"ahler metric $\Omega_{S}=\Omega_0 + \spbp \phi_{S}$ satisfying
\begin{gather}
\frac{1}{(1+4\epsilon_1)^2}\Omega_{S}^{n-1} - \sum_{k=1}^{n} c_k k \chi^{n-k} \Omega_{S}^{k-1}>0.
\label{omegaSsat}
\end{gather}
Here $\epsilon_1<\frac{\epsilon}{3}<1$ is independent of $s$. Note that more precisely, the induction hypothesis, gives the above inequality without the coefficient $(1+4\epsilon_1)^{-2}$, but since the cone condition is pointwise, we can obtain the above inequality, by possibly shrinking $U_S$ and compactness. If $s<s_0'$, then by Proposition \ref{prop:conc-massnonuniform-nondeg}, $S$ contains $\pi(E_{\epsilon_{\ref{secondmaintheorem}}(T_s)}\cup \tilde{Y} \cup \tilde{E})$. Since $\pi^* \chi \leq C_{\chi}  \tilde\Omega_{0}$ for a uniform $C_{\chi}$, we see that Lemma \ref{lem:jackinguptheconecondition} is applicable to $\Omega = \pi^*\Omega_S$ and $\chi=\pi^*\chi$. Therefore, we get a metric $\Omega_{S,s} \in [\pi^*\Omega_S](1+10s)+10s^2 [ \tilde\Omega_{0}]$ satisfying 
\begin{gather}
\frac{1}{(1+2\epsilon_1)^2}\Omega_{S,s}^{n-1} - \sum_{k=1}^n c_k k\chi^{n-k} \Omega_{S,s}^{k-1}>0,
\label{ineq:omegaSsat}
\end{gather}
on $\pi^{-1}(U_S)$. Choose $s$ to be (possibly smaller than earlier) that in addition to all the conditions above, the following condition is met.
\begin{gather}
(1+10s+10s^2)^n-1 <\frac{\epsilon_1}{200}.
\label{ineq:sfixed}
\end{gather}
We do not require any more conditions on $s$ and hence we stick with this choice for the remainder of this section.\\
\indent Let $\phi_{\Theta}$ be a quasi-plurisubharmonic function on $\tilde{Z}$ such that $\Theta=\tilde\Omega_0+\spbp \phi_{\Theta}$. Now denote $\phi_{T_{s}}=(\phi_{\Theta} -\beta' \ln \vert \pi^* \mathcal{S}_Y \vert_{h^N} ^2 +\beta'' \ln \vert s_E \vert^2_{\otimes h_i})(1+10s)$. Define a plurisubharmonic function $\phi^i_{\delta,s}$ as the $\delta$-smoothing of $\phi_{T_{s}}+(1-100\epsilon) \phi^i_{\tilde\Omega_{0}}$. It is clearly well-defined on $\bar{B}_{9r/5}^i$. 
\begin{remark}\label{rem:Aboutstronstrict}
For all sufficiently small $0<\delta<1$ such that there exists a constant coefficient form $\chi_{0}$ in $B_{\delta}(x)$ satisfying $$\chi_{0}\leq \tilde\chi_{s} \leq \chi_{0} \left (1+\frac{\epsilon_1}{1000000n^{10n}} \right ),$$ we see (using the proof of Proposition \ref{prop:conc-massnonuniform} for instance) that the $\phi^i_{\delta,s}$ satisfies the strict cone condition 
\begin{gather}
\frac{1}{(1+2\epsilon_1)^2} (\spbp \phi^i_{\delta,s})(x)^m - \displaystyle \sum_{j=0}^{m-1} b_j \tilde\chi_s^{m-j}(\spbp \phi^i_{\delta,s})^j (x) >0,
\label{strongstrictagain}
\end{gather}
for all $x\in \tilde{Y}^c \cap \tilde{E}^c$ away from say, $4\delta$-neighbourhoods (measured using $\tilde\Omega_{0}$) from $\tilde{Y}\cup \tilde{E}$.
\end{remark}
\indent Choose $\delta$ sufficiently small so that for every $x\in \tilde{Z}$, there exists a $\chi_{0}$ as required in Remark \ref{rem:Aboutstronstrict} on $B_{2\delta}(x)$. 
 We wish to glue (a small modification) of the metric $\Omega_{S,s}$ with the $\phi^i_{\delta,s}$ as before. Choose $\delta$ to be small enough so that on all charts $B_i$, the following holds.
\begin{gather}
\tilde{\Omega}_{s}-(\tilde{\Omega}_{s})_{,\delta} +(1-100\epsilon)\tilde\Omega_{0,\delta}- \tilde\Omega_{0} \geq -150\epsilon \tilde\Omega_{0,\delta},
\label{smallnessofdelta}
\end{gather} 
\indent Now we need to prove that the ``refined" Lelong numbers $\nu(x,\delta)$ of $T$ are somewhat large near $\tilde{Y}\cup \tilde{E}$. Indeed, the singularity of $T$ near $\tilde{E}$ is an explicit logarithmic one. The one near $\tilde{Y}$ is not so explicit. However, Lemma \ref{lem:largelelong} implies that for sufficiently small $\delta$, the $\delta$-Lelong number $\nu^i(x,\delta)$ of $\phi_{T_s} + (1-100\epsilon) \phi^i_{\tilde\Omega_{0}}$ is larger than $5\epsilon_{\ref{secondmaintheorem}}$ in a neighbourhood $O_{\delta} \subset \bar{O}_{\delta} \subset \pi^* U_S \cap \tilde{Z}$ of $\tilde{Y}\cap \tilde{E}$. Choose $\delta$ to be so small that a $10\delta$-neighbourhood (in the $\tilde\Omega_{0}$ sense) of $\tilde{Y} \cup \tilde{E}$ is contained in $O_{\delta}$.\\
\indent We can now decrease $\delta$ further (if necessary) and define $\psi_{s}$ to be the regularised maximum (at the level $\frac{\epsilon r^2}{3})$ over the balls $\bar{B}_{9/5}^i$ of the functions $\phi^i_{\delta,s}-\phi^i_{\tilde\Omega_{0}}$ and $\phi_{S,s} + 3.5 \ln \delta$. Using Proposition 4.1 of \cite{gchen}, we see that $\psi_{s}$ is smooth and coincides with $\phi_{S,s} +3.5 \ln \delta$ in $O_{\delta}$. Moreover,
\begin{gather}
\spbp(\phi^i_{\delta,s}-\phi^i_{\tilde\Omega_{0}})+\tilde{\Omega}_{s} =  \spbp \phi^i_{T_{s},\delta}  +\tilde{\Omega}_{s} + (1-100\epsilon)\spbp \phi^i_{\tilde\Omega_{0},\delta}- \tilde\Omega_{0} \nonumber \\
=(T_{s})_{,\delta} + \tilde{\Omega}_{s}-(\tilde{\Omega}_{s})_{,\delta} +(1-100\epsilon)\tilde\Omega_{0,\delta}- \tilde\Omega_{0} \nonumber \\
\geq (T_{s})_{,\delta}-150\epsilon \tilde\Omega_{0,\delta}.
\end{gather}
  By assumption on the smallness of $\delta$, it is a K\"ahler current throughout $\tilde{Z}$, satisfies the cone condition away from $O_{\delta}$. Hence, $\tilde\Omega_{\psi_{s}}=\tilde\Omega_{s} + \spbp \psi_s$ satisfies  
\begin{gather}
\frac{1}{(1+1.5\epsilon_1)^2} \tilde\Omega_{\psi_{s}}^m - \sum_{j=0}^{m-1} b_j \tilde\chi_s^{m-j} \tilde\Omega_{\psi_{s}}^j>0,
\label{strictconeepsilon1}
\end{gather}
away from $\tilde{E}$ (and in fact, since $\psi$ coincides with $\pi^* \phi_S$ in $O$, these functions descend to smooth functions on $Z$ and satisfy the cone condition on it).   \\ 
\indent Now we follow the proof of Lemma \ref{localmetric}. Extend $\psi$ in the same manner as before to a neighbourhood and consider $\psi_{s}'=\psi_{s} + Cdist(.,\tilde{Z})^2$ in a small neighbourhood $W_1$ of $\tilde{Z}$ in $\tilde{M}$ such that $\tilde{\Omega}_{W_1} = \tilde{\Omega}_{s}+\spbp \psi_{s}'$ satisfies the following condition.
\begin{gather}
\frac{1}{(1+\epsilon_1)^2}n\tilde{\Omega}_{W_1}^{n-1}- \displaystyle \sum_{k=1}^{n-1}c_k k \tilde{\chi}_s^{n-k}\tilde{\Omega}_{W_1}^{k-1}\geq 0.
\label{ineq:w1satisfaction}
\end{gather} 
\indent The construction of Lemma \ref{lem:jackinguptheconecondition} shows the existence of a metric $\tilde{\Omega}_{W_2}$ on a neighbourhood $W_2$ of $E$ satisfying the same inequality as in \ref{ineq:w1satisfaction}. As in \cite{gchen}, we glue (modifications of) these metrics to obtain a metric $\tilde\Omega_{U}$ on a neighbourhood of $\tilde{Z}\cup E$ satisfying 
\begin{gather}
\frac{1}{(1+0.5\epsilon_1)^2}n\tilde{\Omega}_{U}^{n-1}- \displaystyle \sum_{k=1}^{n-1}c_k k \tilde{\chi}_s^{n-k}\tilde{\Omega}_{U}^{k-1}\geq 0.
\label{ineq:usatisfaction}
\end{gather} 
On, $W_1' = \pi(\tilde{U}\cap \tilde{E}^c)$, the function $\psi'_{W_1',s}=\psi'_{s}\circ \pi^{-1}$ is bounded. We define $$\psi''_{s} := \frac{\psi'_{W_1',s}}{1+10s+10s^2}.$$ Since $s$ satisfies \ref{ineq:sfixed}, we are done by the arguments of the proof of Lemma \ref{localmetric}. \\
\indent Given Proposition \ref{prop:nonuniformlocal}, we can follow the proof of Theorem \ref{thm:main} to complete the proof of Theorem \ref{thm:main2}. \qed

\end{document}